\documentclass[11pt]{article}
\usepackage{amsmath,amscd}
\usepackage{amstext}
\usepackage{amssymb,epsfig}
\usepackage{amsthm}
\usepackage{amsfonts}
\usepackage{graphicx}
\usepackage{fancybox, floatrow}
\usepackage[all]{xy}
\usepackage{tikz}
\usetikzlibrary{arrows,automata}

\renewcommand{\footnote}{\endnote}

\usepackage{geometry}      
\newtheorem{theoreme}{Theorem}[section]

\newtheorem{lemme}[theoreme]{Lemma}
\newtheorem{proposition}[theoreme]{Proposition}
\newtheorem{corollary}[theoreme]{Corollary}

\theoremstyle{definition}
\newtheorem{definition}[theoreme]{Definition}
\newtheorem{remarque}[theoreme]{Remark}
\newtheorem{corollaire}[theoreme]{Corollary}
\newtheorem{example}[theoreme]{Example}

\newtheorem{Corollary}{Corollary}

\usetikzlibrary{shapes,patterns}
\usepackage[utf8]{inputenc}      % pour les accents

\newcommand{\comment}[1]{ \begin{center} \ovalbox{ {\sf #1}  }\end{center}} 

\def\epsilon{\varepsilon}
\def\R{{\mathcal R}^{\mbox{\scriptsize ind}}}
\def\C{{\mathcal C}^{\mbox{\scriptsize ind}}}
\def\X{{\mathcal X}^{\mbox{\scriptsize ind}}}
\def\T{T^{\mbox{\scriptsize ind}}}
\def\S{\boldsymbol{S}}
\def\M{{\boldsymbol{M}}}
\def\A{{\boldsymbol{A}}}
\def\r{\boldsymbol{r}}
\def\m{\boldsymbol{m}}
\def\N{\boldsymbol{N}}

\title{An example of PET. Computation of the Hausdorff dimension of the aperiodic set.\thanks{\em This work has been supported by the Agence Nationale de la Recherche -- ANR-10-JCJC 01010.
We would like to acknowledge Tom Schmidt and Vincent Delecroix for many discussions at different steps of the paper.}}

\author{
Nicolas B\'edaride
	\footnote{Aix Marseille Universit\'e, CNRS, Centrale Marseille, I2M, UMR 7373, 13453 Marseille, France. 
	Email address: nicolas.bedaride@univ-amu.fr}
, Jean-Fran\c cois Bertazzon
	\footnote{Lyc\'ee Notre-Dame de Sion, 231 Rue Paradis, 13006 Marseille, France.
	Email address: jeffbertazzon@gmail.com}
}

\begin{document}

\newcommand{\floor}[1]{{\left\lfloor #1 \right\rfloor}}

\date{}
\maketitle

\begin{abstract}
We introduce a family of piecewise isometries. This family is similar to the ones studied by Hooper and Schwartz. We prove that a renormalization scheme exists inside this family and compute the Hausdorff dimension of the discontinuity set. The methods use some cocycles and a continued fraction algorithm.
\end{abstract}

%\tableofcontents
%%%%%%%%%%%%%%%%%%%%%%%%%%%%%%%%%%%%%%%%%%%%
%%%%%%%%%%%%%%%%%%%%%%%%%%%%%%%%%%%%%%%%%%%%
\section{Introduction}
%%%%%%%%%%%%%%%%%%%%%%%%%%%%%%%%%%%%%%%%%%%%
%%%%%%%%%%%%%%%%%%%%%%%%%%%%%%%%%%%%%%%%%%%%

\subsection{Background}
A piecewise isometry in $\mathbb{R}^n$ is defined in the following way: consider a finite set of hyperplanes, the complement $X$ of their union has several connected components. The piecewise isometry is a map $T$ from $X$ to $\mathbb{R}^n$ locally defined on each connected set as an isometry of $\mathbb{R}^n$. Now consider the pre-images of the union of the hyperplanes by $T$: it is a set of zero Lebesgue measure. Thus almost every point of $X$ has an orbit under $T$ and we will study this dynamical system $(X_0,T)$ where $X\setminus X_0$ is of zero measure.
This class of maps has been well studied in dimension one with the example of the interval exchange maps, see \cite{Fer.13}: the map is bijective, equal to the identity outside a compact interval and the isometries which locally define $T$ are translations. Remark that the case with non oriented interval exchange is more difficult (and called interval exchange with flips). The strict case of the dimension two began ten years ago with the paper of  \cite{Ad.Ki.Tr.01}. Since them, different examples have been examined in order to exhibit different behaviors, see for example \cite{As.Go.04}. The first general result has been obtained by Buzzi, proving that every piecewise isometry has zero entropy, see \cite{Buzz.01}.
 An important class of piecewise isometries is the outer billiard. This map has known a lot of developments in recent years with the work of Schwartz: \cite{Schw.09}, \cite{Schwartz.10},  \cite{Schwartz.12} and \cite{Schwartz.pet.12}. He describes the first example of a piecewise isometry of the plane with an unbounded orbit. In his recent papers he defines a new type of piecewise isometry, called Polytope Exchange Transformation (PET for short) and shows that these maps describe the compactification of the outer billiard outside a kite. Independently, Hooper has studied renormalization in some piecewises isometries, see \cite{Hoop.13}, and in \cite{Hooper} shows some pictures of a discontinuity set which seems very close to Schwartz's work. Finally let us mention the work of \cite{Brlek.Mass.Labbe.Mendes.12} which is also close to our example.
 
 In the present paper we describe a family of piecewise isometries. We prove that a renormalization scheme exists inside this family and compute the Hausdorff dimension of the discontinuity set.

\subsection{An informal definition of our dynamical system}
Here we study a dynamical system closely related to a PET. We want to exchange one square and one rectangle. Let $\mathcal C$ be the square $[0,1]^2$, and let $\theta \in [0,1[$ be a real number, consider the rectangle $\mathcal R_\theta=[1,1+\theta]\times [0,1]$ and let $\mathcal X_\theta=\mathcal C\cup \mathcal R_\theta$. See the following figure:

\begin{center}
\begin{tikzpicture}
\fill[blue!30] (0,0) -- (2,0) -- (2,2) -- (0,2)  -- cycle; \fill[red!50] (2,0) -- (3.4,0) -- (3.4,2) -- (2,2)  -- cycle;
\node[below]  at (0,0) { $0$};\node[below] at (2,0) { $1$}; \node[below] at (3.4,0) { $1+\theta$};
\node[left]  at (0,0) { $0$};\node[left] at (0,2) { $1$};  \node at (1,1) {\large  ${\mathcal C}$}; \node at (2.7,1) {\large ${\mathcal R}_{ \theta}$};
\draw (0,0) -- (2,0) -- (2,2) -- (0,2)  -- cycle; \draw (2,0) -- (3.4,0) -- (3.4,2) -- (2,2)  -- cycle;
\end{tikzpicture}

{Definition of $\mathcal X_\theta$.}
\end{center}

We want to define a piecewise isometry which globally exchanges the square and the rectangle as described below:

\begin{center}
\raisebox{-1.5cm}{
\begin{tikzpicture}
\fill[blue!30] (0,0) -- (2,0) -- (2,2) -- (0,2)  -- cycle; \fill[red!50] (2,0) -- (3.4,0) -- (3.4,2) -- (2,2)  -- cycle;
\node[below]  at (0,0) { $0$};\node[below] at (2,0) { $1$}; \node[below] at (3.4,0) { $1+\theta$};
\node[left]  at (0,0) { $0$};\node[left] at (0,2) { $1$}; 
\draw (0,0) -- (2,0) -- (2,2) -- (0,2)  -- cycle; \draw (2,0) -- (3.4,0) -- (3.4,2) -- (2,2)  -- cycle;
\end{tikzpicture}
}
$\xrightarrow[  \qquad \qquad ]{} $
\raisebox{-1.5cm}{
\begin{tikzpicture}
\fill[blue!30] (1.4,0) -- (3.4,0) -- (3.4,2) -- (1.4,2)  -- cycle; \fill[red!50] (0,0) -- (1.4,0) -- (1.4,2) -- (0,2)  -- cycle;
\node[below]  at (0,0) { $0$};\node[below] at (2,0) { $1$}; \node[below] at (1.4,0) { $\theta$};
\node[left]  at (0,0) { $0$};\node[left] at (0,2) { $1$}; 
\draw (1.4,0) -- (3.4,0) -- (3.4,2) -- (1.4,2)  -- cycle; \draw (0,0) -- (1.4,0) -- (1.4,2) -- (0,2)  -- cycle;
\end{tikzpicture}
}
\end{center}

For a fixed $\theta$, there are exactly $24$ transformations which isometrically exchange these two pieces
and some of them are conjugated via the orthogonal reflection through  $y=1/2$.
In fact there are only two of them for which the dynamical behavior is interesting and  we will consider them in this paper. 
These two maps are parametrized by $\omega=(\theta,\varepsilon)$ with $\epsilon=\pm 1$.

\subsection{Outline }
In Section \ref{sec:defmap} we give a precise definition of our dynamical system denoted $(\mathcal X_\omega, T_\omega)$. We define the coding of this map and the associated symbolic dynamical system. 
In Section \ref{sec:induction}, we introduce the renormalization and we study the map $T_\omega$ and show that an induction process exists: 
	there exists a subset of $\mathcal X_\omega$ where the first return map of $T_\omega$ is conjugated to $T_{S(\omega)}$ for some map $S$. 
Then in Section \ref{se:autosimilaire} we study the map $S$. This map can be seen as a continued fraction algorithm. We study this continued fraction in Proposition \ref{Squadratique} and compute an invariant measure. These results can be seen as an application of the theory developed by Arnoux-Schmidt in \cite{Arn.Sch.13}.
In Section \ref{subsec:dynamic}, we describe the relation between the dynamics on the aperiodic set, and a Sturmian subshift.
The next step is to obtain a formula for the Hausdorff dimension of the set of aperiodic points in $(\mathcal X_\omega, T_\omega)$. The idea is to give a formula for the Hausdorff dimension in terms of a Lyapunoff exponent of a cocycle.
To obtain it, we need to use Oseledets theorem. The problem is that the invariant measure of $S$ does not have good properties. 
Thus we need to prove that an accelerated map is ergodic, see Section \ref{sec:Sacc}. 
Then in Section \ref{sec:calc-hausdorff}, we explain the relation between the Hausdorff dimension and a cocycle. 
Finally in Section \ref{sec:numlyap} we deduce some approximations of the Hausdorff dimension. Some technical parts are left for the Appendix. 

%%%%%%%%%%%%%%%%%%%%%%%%%%%%%%%%%%%%%%%%%%%%
%%%%%%%%%%%%%%%%%%%%%%%%%%%%%%%%%%%%%%%%%%%%
\section{Definition of the dynamical system and first properties}\label{sec:defmap}
%%%%%%%%%%%%%%%%%%%%%%%%%%%%%%%%%%%%%%%%%%%%
%%%%%%%%%%%%%%%%%%%%%%%%%%%%%%%%%%%%%%%%%%%%

\subsection{Definition of the dynamical system} 
%%%%%%%%%%%%%%%%%%%%%%%%%%%
%%%%%%%%%%%%%%%%%

Let us define $\omega=(\theta,\epsilon) \in \Omega=[0,1[ \times \{-1,1\}$, and denote the  segments on  the boundary of the square and the rectangle by ${\mathcal D}_\theta=\partial{\mathcal C}\cup \partial{\mathcal R_\theta}$. Now let us define two maps $f_\epsilon$ from $[0,1]$ into itself by 
$$\begin{array}{ccc}
[0,1]&\rightarrow&[0,1]\\
x &\mapsto& f_\epsilon(x) =\begin{cases} x \quad & \mbox{ if } \epsilon=-1,\\ 1-x \quad & \mbox{ if }  \epsilon=1.\end{cases}
\end{array}$$

Consider the map $T_\omega$ defined by:
$$\begin{array}{ccc}
\mathring{\mathcal C} \cup \mathring{\mathcal R_\theta}&\rightarrow&\mathcal C \cup \mathcal R_\theta\\
z=(x,y)& \mapsto& \begin{cases} \big(1+\theta-y ,f_\epsilon(x) \big) \quad & \mbox{ if } (x,y)\in \mathring{\mathcal C} 	,
	\\ (x-1,1-y) &\mbox{ if } (x,y)\in \mathring{\mathcal R_\theta}.
	\end{cases}
\end{array}$$
%\ \label{predefdeT}

\begin{center}
\raisebox{-1.5cm}{
\begin{tikzpicture}
\fill[blue!30] (0,0) -- (2,0) -- (2,2) -- (0,2)  -- cycle;
\draw (0,0) -- (2,0) -- (2,2) -- (0,2)  -- cycle;
\fill[red!50] (2,0) -- (3.4,0) -- (3.4,2) -- (2,2)  -- cycle;
\draw (2,0) -- (3.4,0) -- (3.4,2) -- (2,2)  -- cycle;
\draw (2,1.5) -- (2.5,1.5) -- (2.5,2);
\draw[->] (0.7,0) arc (0:90:0.7);
\node[below]  at (0,0) { $0$};\node[below] at (2,0) { $1$}; \node[below] at (3.4,0) { $1+\theta$};
\node[left]  at (0,0) { $0$};\node[left] at (0,2) { $1$}; 
\end{tikzpicture}
}
$\underset{\epsilon=-1}{\xrightarrow[  \qquad \qquad ]{} }$
\raisebox{-1.5cm}{
\begin{tikzpicture}
\fill[blue!30] (1.4,0) -- (3.4,0) -- (3.4,2) -- (1.4,2)  -- cycle;
\draw[black] (1.4,0) -- (3.4,0) -- (3.4,2) -- (1.4,2)  -- cycle;
\fill[red!50] (0,0) -- (1.4,0) -- (1.4,2) -- (0,2)  -- cycle;
\draw[black] (0,0) -- (1.4,0) -- (1.4,2) -- (0,2)  -- cycle;
\draw (0,0.5) -- (0.5,0.5) -- (0.5,0);
\draw[->] (3.4,0.7) arc (90:180:0.7);
\node[below]  at (0,0) { $0$};\node[below] at (2,0) { $1$}; \node[below] at (1.4,0) { $\theta$};
\node[left]  at (0,0) { $0$};\node[left] at (0,2) { $1$}; 
\end{tikzpicture}
}
\end{center}

\begin{center}
\raisebox{-1.5cm}{
\begin{tikzpicture}%[scale=.8]
\fill[blue!10] (0,2) -- (2,2) -- (2,0) -- cycle;
\fill[blue!50] (0,0) -- (0,2) -- (2,0)  -- cycle;
\draw (0,0) -- (2,0) -- (2,2) -- (0,2)  -- cycle;
\fill[red!50] (2,0) -- (3.4,0) -- (3.4,2) -- (2,2)  -- cycle;
\draw (2,0) -- (3.4,0) -- (3.4,2) -- (2,2)  -- cycle;
\draw (0,2) -- (2,0);
\draw (2,1.5) -- (2.5,1.5) -- (2.5,2);
\node[below]  at (0,0) { $0$};\node[below] at (2,0) { $1$}; \node[below] at (3.4,0) { $1+\theta$};
\node[left]  at (0,0) { $0$};\node[left] at (0,2) { $1$}; 
\end{tikzpicture}
}
$\underset{\epsilon=1}{\xrightarrow[  \qquad \qquad ]{} }$
\raisebox{-1.5cm}{
\begin{tikzpicture}%[scale=.8]
\fill[blue!10] (1.4,0) -- (1.4,2) -- (3.4,0) -- cycle;
\fill[blue!50] (1.4,2) -- (3.4,2) -- (3.4,0)  -- cycle;
\draw[black] (1.4,0) -- (3.4,0) -- (3.4,2) -- (1.4,2)  -- cycle;
\fill[red!50] (0,0) -- (1.4,0) -- (1.4,2) -- (0,2)  -- cycle;
\draw[black] (0,0) -- (1.4,0) -- (1.4,2) -- (0,2)  -- cycle;
\draw (0,0.5) -- (0.5,0.5) -- (0.5,0);
\node[below]  at (0,0) { $0$};\node[below] at (2,0) { $1$}; \node[below] at (1.4,0) { $\theta$};
\node[left]  at (0,0) { $0$};\node[left] at (0,2) { $1$}; 
\end{tikzpicture}
}
\end{center}

\begin{lemme}
The map $T_\omega$ is the product of the following isometries:
\begin{itemize}
\item For $\epsilon=-1$, the restriction of the map to the square is the product of the rotation of angle $\frac{\pi}{2}$ and center $(1/2,1/2)$ and the translation by $(\theta,0)$. 
\item For $\epsilon=1$, the restriction of the map to the square is the composition of the orthogonal reflection through  $y+x=1$ and the translation by $(\theta,0)$.
\item The restriction of the map to the rectangle does not depend on $\epsilon$: it is the product of the translation by $(-1,0)$ and the orthogonal reflection through $y=\frac{1}{2}$. 
\end{itemize}
\end{lemme}
The proof is left to the reader.

\begin{definition}
We define  {\bf the set of discontinuities} by ${\mathcal D}_{\omega}=\displaystyle\bigcup_{k\in \mathbb{N}}T_\omega^{-k}({\mathcal D}_\theta)$. We also define $\mathcal X_\omega=\mathcal X_\theta \setminus {\mathcal D}_\omega$.
\end{definition}

\begin{lemme}
For every point $z$ in  $\mathcal X_\omega$ the orbit of $z$ under $T_\omega$ is well defined. The set $\mathcal X_\omega$ is of Lebesgue measure $1$.
\end{lemme}

\begin{proof}
First of all, ${\mathcal D}_\omega$ is a countable union of segments, thus it is of zero Lebesgue measure. Now remark that the orbit of a point $z$ under $T_\omega$ is not defined if and only if there exists an integer $n$ such that $T_\omega ^n(z)\in {\mathcal D}_ \theta$.
By definition $\mathcal X_\omega$ is the complement of this set.
\end{proof}

\begin{remarque} \label{re:gdelta}
We can be more precise: let $n$ be an integer, the set $T_\omega^{-n} {\mathcal D}_\theta$ is a finite union of horizontal or vertical segments. We deduce that $\mathcal X_\omega$ is the union of dense open sets. 
\end{remarque}

We introduce a notation for the restriction of the defined sets to the square and the rectangle by: 
\begin{itemize}
\item ${\mathcal D}_\omega^c={\mathcal D}_\omega\cap \mathcal C$ and ${\mathcal D}_\omega^r={\mathcal D}_\omega\cap \mathcal R_\theta$ (they are of zero Lebesgue measure);
\item $\mathcal C_\omega = \mathcal C \cap \mathcal X_\omega$ and $\mathcal R_\omega = \mathcal R_\theta \cap \mathcal X_\omega$.
\end{itemize}

\begin{figure}[H]
\includegraphics[width=14cm]{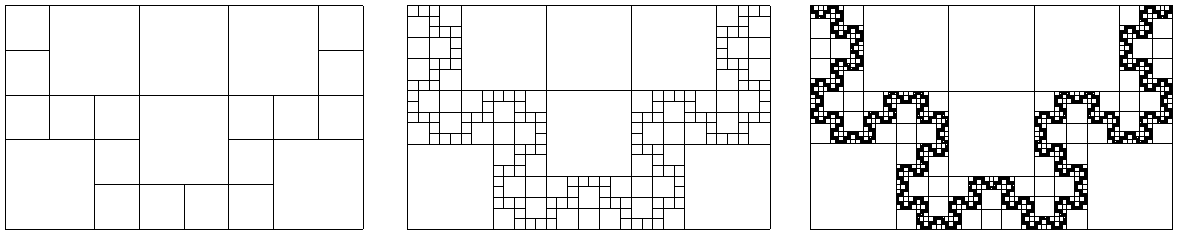}
\caption{Representation of ${\mathcal D}_\omega$ for the following values for $\omega$:
	$\left( \frac{3}{5} ,-1 \right)$, $\left( \frac{13}{21} ,-1 \right)$ and $\left( \frac{\sqrt 5 -1}{2} -1,-1 \right)$.}
\end{figure}

\begin{figure}[H]
\includegraphics[width=14cm]{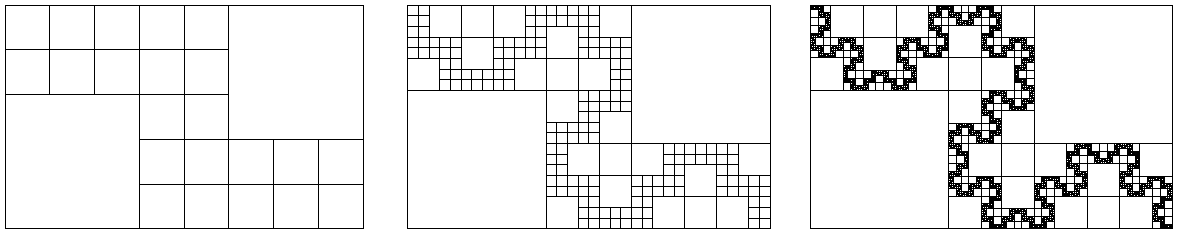}
\caption{Representation of ${\mathcal D}_\omega$ of $\omega = \left( \frac{3}{5} ,1 \right)$, $\omega = \left( \frac{13}{21} ,1 \right)$
and for $\omega = \left( \frac{\sqrt 5 -1}{2} -1,1 \right)$.}
\end{figure}

\begin{proposition}
We have $T_\omega(\mathcal X_\omega)=\mathcal X_\omega$. Moreover
the dynamical system $(\mathcal X_\omega,T_\omega)$ is a bijective piecewise isometry, with the Lebesgue measure as an invariant measure.
\end{proposition}

\begin{proof}
Each isometry involved in the definition of $T_\omega$ is clearly a bijection. 
The following map from $T_\omega(\mathcal X_\omega)$ into $\mathcal X_\omega$ is thus well defined:
$$
(x,y) \mapsto T^{-1}_\omega(x,y) =\begin{cases} \big(f_\epsilon(y),1+\theta-x \big)  \quad & \mbox{ if } (x,y)\in T_\omega\big(\mathcal C_\omega \big)  \\
	 (x+1,1-y) \quad & \mbox{ if }  (x,y)\in T_\omega\big(\mathcal R_\omega\big)\end{cases}
$$
To prove the result we just have to analyze the discontinuity set of this map and prove that it coincides with $\mathcal D_\omega$.
We define ${\mathcal D}_\theta^{-1}=\partial{T_\omega(\mathcal C)}\cup \partial{T_\omega(\mathcal R_\theta)}$, ${\mathcal D}_\omega^{-1}=\bigcup_{n\in\mathbb{N}}T_\omega^{-n}({\mathcal D}_\theta^{-1})$  and $\mathcal F_\omega$ as the union ${\mathcal D}_\theta\cup{\mathcal D}_\theta^{-1}$.
It is clear that $z\in {\mathcal D}_\omega$ (resp. $z\in{\mathcal D}_\omega^{-1})$ if and only if there exists an integer $n$ such that $T_\omega^n z \in {\mathcal D}_\theta$ (resp. $T_\omega^n z \in  {\mathcal D}_\theta^{-1}$).

Remark that $\mathcal F_\omega$ is invariant by the map \[	\mbox{Sym}_{\omega} : (x,y) \mapsto \big(1+\theta-x,f_\epsilon(y) \big).	\]
Since this map fulfills 
$T_\omega\circ Sym_\omega=Sym_\omega\circ T_\omega^{-1}$ we deduce that $T_\omega(\mathcal X_\omega)=\mathcal X_\omega$. 
The rest of the proof is easy since all the maps are isometries and therefore the Lebesgue measure is an invariant measure.
\end{proof}

\subsection{Symbolic dynamics of a piecewise isometry} 	
%%%%%%%%%%%%%%%%%%%%%%%%%%%
%%%%%%%%%%%%%%%%%%%%%%%%%%%

We need to introduce some notions of symbolic dynamics, see \cite{Pyth.02}.
Let $\mathcal{A}$ be a finite set called  alphabet, {\bf a word} is a finite string of elements in $\mathcal{A}$, its length is the number of elements in the string. The set of all finite words over $\mathcal{A}$ is denoted $\mathcal{A}^*$. 
A (one sided) sequence of elements of $\mathcal{A}$, $u=(u_n)_{n\in\mathbb{N}}$ is called an infinite word. A word $v_0\dots v_k$ appears in $u$ if there exists an integer $i$ such that $u_{i}\dots u_{i+k+1}=v_0\dots v_{k}$. 
If $u=u_0\ldots u_n$ is a finite word, we denote by $\bar u$ the infinite word $u_0\ldots u_n u_0 \ldots u_n \ldots$. This word is {\bf periodic} of period $u$.
For an infinite word $u$, {\bf the language} of $u$ (respectively the language of length $n$)  is the set of all words (respectively all words of length $n$) in $\mathcal{A}^*$ which appear in $u$. We denote it by $L(u)$ (respectively by $L_{n}(u)$). 
We endow the set of sequences $\mathcal A^{\mathbb N}$ with the product topology, then we define for a finite word $v$, the {\bf cylinder} 
$[v]=\{vu, u\in \mathcal A^{\mathbb N}\}$. The set of cylinders form a basis of clopen sets for the topology.

A {\bf substitution} $\sigma$ is an application from an alphabet $\mathcal{A}$ to the set $\mathcal{A}^*\setminus\{\varepsilon\}$ of nonempty finite words on $\mathcal{A}$. 
It extends to a morphism of $\mathcal{A}^*$ by concatenation, that is $\sigma(uv)=\sigma(u)\sigma(v)$.

Now we define an application $\phi:\mathbb{R}^2\mapsto \{a;b\}$ by 
\begin{center} $\phi(x,y) = a$ if $x<1$ \quad and \quad $\phi(x,y)=b$ otherwise. \end{center}
For $\omega=(\theta,\epsilon)\in \Omega$ and $z\in \mathcal X_\omega$ we have,
\begin{center} $\phi(z)=a \Longleftrightarrow z\in \mathcal C_\omega$ \quad and \quad $\phi(z)=b \Longleftrightarrow z\in \mathcal R_\omega$. \end{center}
Let $\Phi_\omega:\mathcal X_\omega\mapsto \{a;b\}^\mathbb{N}$ be {\bf the coding map} defined by
\begin{center} $\Phi_\omega(z)=(u_n)_{n\in\mathbb{N}}$  \quad such that  \quad $u_n = \phi \big( T_\omega^n(z) \big)$. \end{center}

The image by the coding map of the points in $\mathcal X_\omega$ defines a language. 
For a finite (or infinite) word $u$ in this language, a {\bf cell} is the set of points which are coded by this word:
$\mathcal O_u = \{z\in \mathcal X_\omega, \Phi_\omega(z)= u\}$.
The cells $\mathcal O_{\bar u}$, where $u$ is a finite word, are called  {\bf periodic cells} and the period is defined as the period of the word $\bar u$.

\begin{lemme} \label{le:ileper}
If $z$ is a periodic point of $T_\omega$ then $\Phi_\omega(z)$ is a periodic word and the cell is a rectangle.
The restriction of $T_\omega$ to a periodic cell is either a rotation of angle $-\frac{\pi}{2}, \pi$ or $\frac{\pi}{2}$ or an orthogonal reflexion. In all cases, every point of a periodic cell has a periodic orbit.
\end{lemme}

The reader should take care not to confuse the period of  a periodic cell  and the period of the points.
In Corollary \ref{coro:renormalisation} we will see that this result can be improved.
\begin{proof}
Let $z$ be such that $T_\omega^n(z)=z$, and let $v=\Phi(z)$. By definition of $z$, the word $v$ is periodic of period $n$. Let us denote by $u=u_0\dots u_{n-1}$ its period, thus we have $v=u\dots u\dots$.
We define $\mathcal P_a = \mathcal C_\omega$ and $\mathcal P_b = \mathcal R_\omega$.
We deduce $\mathcal O_{\bar u}=\displaystyle\bigcap_{0\leq k\leq n-1}T_\omega^{-k}P_{u_k}$.
Remark that $T_\omega^{-k} P_{u_k}$ is a convex set for every integer $k$, then $\mathcal O_{\bar u}$ is a decreasing intersection of convex set, thus it is a convex set.
Moreover ${\mathcal D}_\theta$ is the union of vertical and horizontal segments, then its image by $T_\omega$ is an horizontal or vertical edge, thus every segment of ${\mathcal D}_\omega$ is horizontal or vertical: we conclude that $\mathcal O_{\bar u}$ is a rectangle. Moreover the restriction of $T_\omega$ to this rectangle has a periodic point.
\end{proof}

\begin{definition}
For $\omega \in \Omega$ and $n\in\mathbb N$, let us denote by $\mathcal I_\omega(n)$ the union of periodic cells of period less than $n$.
\end{definition}

Consider the dynamical system $(\mathcal X_\omega, T_\omega)$. Denote the {\bf periodic points} in this system by $\mathcal I_\omega$. A natural question seems to look at the complement of this subset. 
It is non empty, but we do not know its size. Can we compute it? This question appears naturally in several papers as \cite{Schw.09}. 
Thus we define for each positive integer $n$:
$$
\mathcal K_\omega(n)= \overline{\mathcal X_\omega \setminus \mathcal I_\omega(n)}
	 \quad \text{and} \quad 
\mathcal K_\omega =\displaystyle\bigcap_{n\in\mathbb N} \mathcal K_\omega(n).$$

As above we also define

\[
	\mathcal K_\omega^r(n)= \mathcal K_\omega(n)\cap \mathcal R_\theta,  \quad
	\mathcal K_\omega^c(n)=\mathcal K_\omega(n)\cap \mathcal C,\quad
	 \mathcal K_\omega^r = \displaystyle\bigcap_{n\in\mathbb N} \mathcal K_\omega^r(n) 
	 	\quad \mbox{and} \quad
\mathcal K_\omega^c=\displaystyle\bigcap_{n\in\mathbb N}\mathcal K_\omega (n)^c.
\]

From Remark \ref{re:gdelta} we have $\mathcal K_\omega(n)= \mathcal X_\theta  \setminus \mathring{\mathcal I}_\omega(n)$.
The set $\mathcal K_\omega$ is called the {\bf aperiodic set}.

\begin{figure}[H]
\includegraphics[width=14cm]{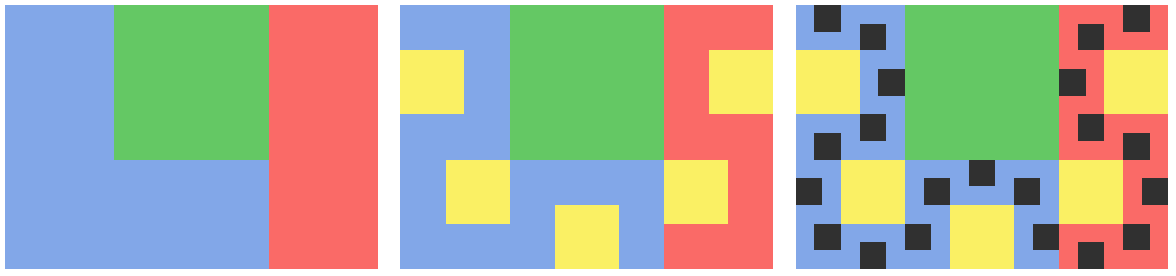}
\caption{Example for $\omega = \left( \sqrt 2-1,-1 \right)$. The sets $\mathcal K_\omega^c(n)$ are included in the blue part and $\mathcal K_\omega^r(n)$ is inside the red part. The periodic island of period one is in green, the islands of period five are in yellow, and in black are the periodic islands of period $21$.}\label{fig:iles}
\end{figure}

\begin{lemme}\label{lem:periode1}
$\;$
\begin{itemize}
\item For $\epsilon=-1$, the set $\mathcal I_\omega(1)$ is a square: its Lebesgue measure is equal to $(1-\theta)^2$ and 
the coding associated to a point of $\mathcal I_\omega(1)$ is $\bar a$.

\item For $\epsilon=1$, the set $\mathcal I_\omega(1)$ is empty, and $\mathcal{I}_\omega(2)$ is the union of two squares. Its area is equal to $2\theta^2$ and the coding associated to the points of $\mathcal I_\omega(2)$ is $\overline{ ab}$ or $\overline{ba}$.
\end{itemize}
\end{lemme}
\begin{proof}
The proof is left to the reader. See the following figure that represents in blue $\mathcal I_\omega(1)$ for $\epsilon=-1$ and $\mathcal I_\omega(2)$  for $\epsilon=1$. 

\hfill
\begin{tikzpicture}
\fill[blue!30] (1.4,1.4) -- (2,1.4) -- (2,2) -- (1.4,2)  -- cycle; 
\node[below]  at (0,0) { $0$};\node[below] at (2,0) { $1$}; \node[below] at (3.4,0) { $1+\theta$};\node[below] at (1.4,0) { $\theta$};
\node[left]  at (0,0) { $0$};\node[left] at (0,2) { $1$};\node[left] at (0,1.4) { $\theta$};
\draw (0,0) -- (2,0) -- (2,2) -- (0,2)  -- cycle; 
\draw (2,0) -- (3.4,0) -- (3.4,2) -- (2,2)  -- cycle;
\draw (1.4,1.4) -- (2,1.4) -- (2,2) -- (1.4,2)  -- cycle; 
\draw[dashed] (1.4,0) -- (1.4,1.4); 
\draw[dashed] (0,1.4) -- (1.4,1.4); 
\end{tikzpicture}
\hfill
\begin{tikzpicture}
\fill[blue!30] (0,0) -- (1.4,0) -- (1.4,1.4) -- (0,1.4)  -- cycle; 
\fill[blue!30] (2,0.6) -- (3.4,0.6) -- (3.4,2) -- (2,2)  -- cycle; 
\node[below]  at (0,0) { $0$};\node[below] at (2,0) { $1$}; \node[below] at (3.4,0) { $1+\theta$};\node[below] at (1.4,0) { $\theta$};
\node[left]  at (0,0) { $0$};\node[left] at (0,2) { $1$};\node[left] at (0,1.4) { $\theta$};
\node[right] at (3.4,0.6) { $1-\theta$};
\draw (0,0) -- (2,0) -- (2,2) -- (0,2)  -- cycle; 
\draw (2,0) -- (3.4,0) -- (3.4,2) -- (2,2)  -- cycle;
\draw (0,0) -- (1.4,0) -- (1.4,1.4) -- (0,1.4)  -- cycle; 
\draw (2,0.6) -- (3.4,0.6) -- (3.4,2) -- (2,2)  -- cycle; 
\end{tikzpicture}
\hfill
\qedhere
\end{proof}

\begin{lemme} \label{lemme}
Let $\omega \in \Omega$.
\begin{enumerate}
\item The point $z$ belongs to $\mathcal K_\omega$ if and only if $z$ has a non periodic orbit under $T_\omega$. 
\item $\mathcal K_\omega \subset \bar {\mathcal D}_\omega$.
\end{enumerate}
\end{lemme}

\begin{proof}
The first point can be deduced from  Lemma \ref{le:ileper}.

Let $z\in \mathcal K_\omega \cap \mathcal X_\omega$. We will show that $z$ belongs to ${\mathcal D}_\omega$. We argue by contradiction: suppose we have
$\inf \{ d(T^n z, {\mathcal D}_\theta)\mid n\in \mathbb N\}=\delta>0$., and consider a square centered in $z$ of diameter $\delta/2$. The orbit of this square never intersects a segment of ${\mathcal D}_\omega$. Thus every point inside this square has the same coding, and this coding is non periodic by assumption. 
Thus the cell of a non periodic word has non empty interior which is impossible.
\end{proof}

To finish this section let us briefly sum up the notations used here:
\begin{center}
$\begin{array}{|c|c|c|}
\hline
{\mathcal D}_\omega&\mathcal{I}_\omega&\mathcal{K}_\omega\\
\hline
\text{Discontinuities}& \text{Periodic points}& \text{Aperiodic points}\\
\hline
\end{array}$
\end{center}
The notation $\mathcal{D}_\omega(n)$ will be used for the same object restricted to the $n$ first elements of an orbit.

\subsection{Results of the paper}
%%%%%%%%%%%%%%%%%%%%%%%
%%%%%%%%%%%%%%%%%%%%%%%
To start with, we prove the following:
\begin{itemize}
\item The map $T_\omega$ has a renormalization scheme (Section \ref{sec:induction} Proposition \ref{induction}).	
\item The dynamical system $(\mathcal K_\omega, T_\omega)$ is conjugated to a rotation of angle $\theta$ (Section \ref{subsec:dynamic} Proposition \ref{prop:rot-aperiodique}).
\end{itemize}
Moreover we compute the Hausdorff dimension of the aperiodic set, and show:

\begin{theoreme}\label{thm:lyap}
There exists a real number $\boldsymbol s$ such that for almost all $\omega$, the Hausdorff dimension of  $\mathcal K_\omega$ is equal to $\boldsymbol s$.
\end{theoreme}

\begin{Corollary}
$\;$
\begin{itemize}
\item We have $1.07 \leq \boldsymbol s \leq 1.55$.
\item Moreover we obtain for $n\in\mathbb N$ :
\begin{equation*} \label{eq:dimcaspointfixe}
\dim_{H}(\mathcal K_\omega)=\begin{cases}
-\dfrac{\ln(2n+\sqrt{4 n^2+1})}{\ln\big( \sqrt{n^2+1}-n \big)} & \mbox{ if } \omega = \Big( \sqrt{n^2+1}-n,-1\Big), \\
-\dfrac{\ln(2n+1+2\sqrt{n^2+n})}{\ln \big( n+1-\sqrt{n(n+2)} \big)} & \mbox{ if } \omega= \Big(\sqrt{n(n+2)}-n,1\Big).
\end{cases}
\end{equation*}
\end{itemize}
\end{Corollary}

This number $\boldsymbol s$ is obtained via the Lyapunoff exponent of a cocycle introduced in Section \ref{sec:calc-hausdorff}.
Remark that the computation of the Hausdorf dimension can be seen as a generalization of the classical case where the fractal set is the solution of an iterated function system.

%%%%%%%%%%%%%%%%%%%%%%%%%%%%%%%%%%%%%%%%%%%%
%%%%%%%%%%%%%%%%%%%%%%%%%%%%%%%%%%%%%%%%%%%%
\section{Induction}\label{sec:induction}
%%%%%%%%%%%%%%%%%%%%%%%%%%%%%%%%%%%%%%%%%%%%
%%%%%%%%%%%%%%%%%%%%%%%%%%%%%%%%%%%%%%%%%%%%

\subsection{Notations}\label{def:S}
%%%%%%%%%%%%%%%%%%%%%%%%
%%%%%%%%%%%%%%%%%%%%%%%%%

Consider $\omega = (\theta,\epsilon)\in ]0,1[\times \{-1,1\}$, we denote by $n_{\omega} = \floor{ \frac{1}{f_\epsilon(\theta)} }$ where $\floor{.}$ denotes the floor function.
We define a map $S$ by
\begin{equation}
\begin{array}{ccc}
 ]0,1[\times \{-1,1\}&\rightarrow&\Omega\\ 
(\theta,\epsilon)&\mapsto&
S(\theta,\epsilon) = \left( \frac{1}{f_\epsilon(\theta)} - n_\omega,  (-1)^{n_\omega+1}  \right).
\end{array}
\end{equation}\label{equ:S}

Now let $\psi_\omega$ be a similitude defined by 
$$\begin{array}{ccc}
\mathbb{R}^2&\rightarrow&\mathbb{R}^2\\
(x,y)&\mapsto&\psi_{(\theta,-1)}(x,y) = \frac 1 \theta (y,x)\\
(x,y)&\mapsto& \psi_{(\theta,1)}(x,y) =\frac 1 {1-\theta} (x,y-\theta) .
\end{array}
$$
Remark that this similitude has a ratio $r(\omega) =\frac{1}{f_\epsilon(\theta)}$ and that the inverse of this map is given by 
$$ \psi_{(\theta,-1)}^{-1} (x,y) = \theta(y,x) \qquad \mbox{ and } \qquad \psi_{(\theta,1)}^{-1}  (x,y)= \big( (1-\theta)x,\theta+y(1-\theta) \big) .$$

The main objects of the paper are the following sets: the notation will be explained in Proposition \ref{induction}. 

\begin{definition}
For $\omega\in \Omega$ we define the {\bf induction zones} by 
\[
\C_\omega = \psi_\omega^{-1} (\mathcal C_{S\omega}), \quad
\R_\omega = \psi_\omega^{-1} (\mathcal R_{S\omega}) \quad \text{and} \quad 
\X _\omega  = \C_\omega \cup \R_\omega.
\]
 \end{definition}
Remark that $\psi_\omega(\X_\omega) = \mathcal X_{S \omega}$ by definition.

\begin{figure}[H]
\includegraphics[width=14cm]{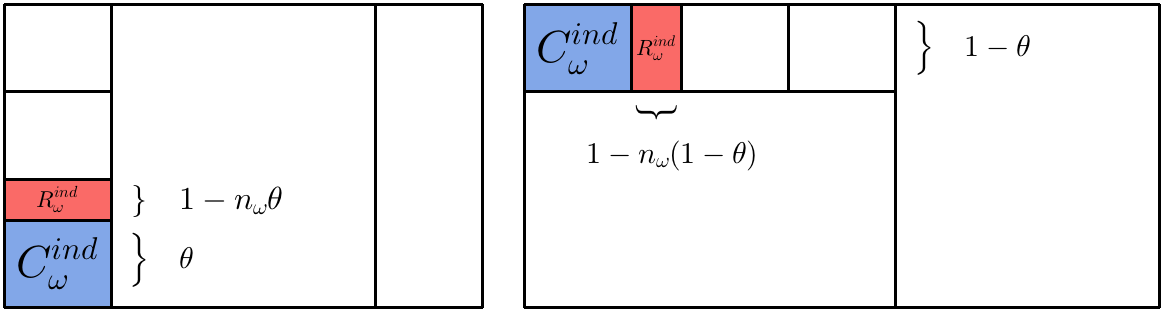}
\caption{Induction zone for the parameters $(\theta,-1)$ and $(\theta,1)$.}
\end{figure}

\subsection{First return map to the induction zone} 
%%%%%%%%%%%%%%%%%%%%%%%
%%%%%%%%%%%%%%%%%%%%%%%

The first return time of $z\in \X_\omega = \C_\omega \cup \R_\omega$ is defined as 
$ {\sf n} (z)=\min\{k\geq 1, T_\omega^k(z)\in \X_\omega \}.$
Then we define $\T_\omega$ as the map of $\X_\omega$ into itself defined as
$z\mapsto T^{{\sf n}(z)}(z)$. It is the first return map of $\T_\omega$ onto $\X_\omega$.

\begin{proposition}[Induction] \label{induction}
The maps ${\sf n}$ and $ \T_\omega$ are well defined on $ \X_\omega$ and we have for all $z\in \mathcal X_{S\omega}$: 
$$\psi_\omega  \circ \T_\omega \circ \psi_{\omega}^{-1}(z) =  T_{S(\omega)}(z).$$
\end{proposition}

\begin{figure}[H]
\includegraphics[width=14cm]{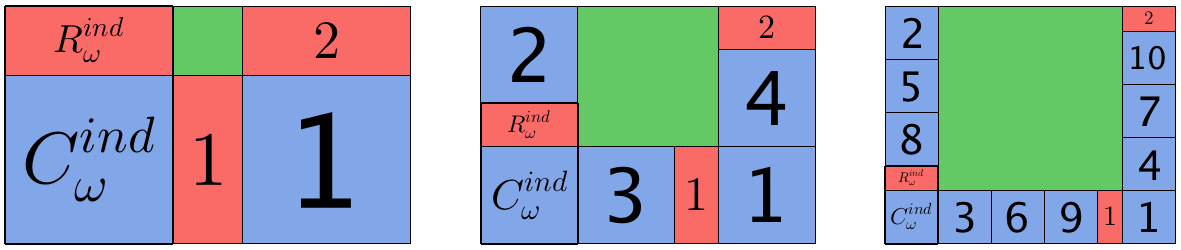}
\caption{Induction scheme for $\epsilon=-1$ with $\theta_n=\dfrac{1}{\sqrt{n(n+1)}}$ where $n=1,2,4$.}
\end{figure}
\begin{figure}[H]
\includegraphics[width=14cm]{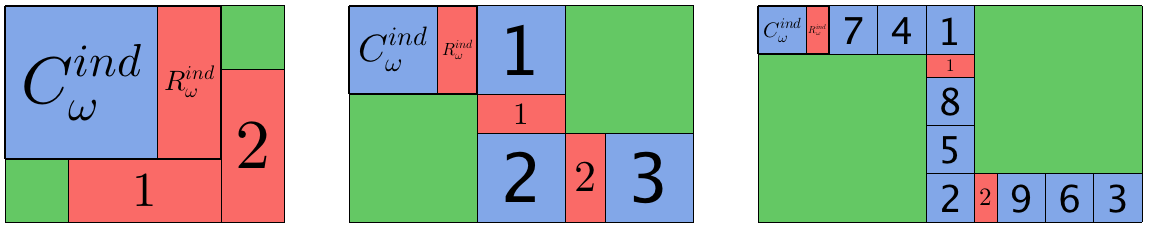}
\caption{Induction scheme for $\epsilon=1$ with $\theta_n=1-\dfrac{1}{\sqrt{n(n+1)}}$ where $n=1,2,4$.}
\end{figure}

%%%%%%%%%%%%%%%%%%%%%%%%%%%%%%%%%%%%%%%%%%%%
%%%%%%%%%%%%%%%%%%%%%%%%%%%%%%%%%%%%%%%%%%%%
\subsection{Induction, proof of Proposition \ref{induction}}\label{induce}
%%%%%%%%%%%%%%%%%%%%%%%%%%%%%%%%%%%%%%%%%%%%
%%%%%%%%%%%%%%%%%%%%%%%%%%%%%%%%%%%%%%%%%%%%

\noindent
$\bullet$ {\it Let us start with $\epsilon=-1$ for a parameter $\theta \in ]0,1[$: } %%

Let $(x,y)$ be a point in the interior of $\C_{\omega}$. This means $0<x<\theta$ and $0<y<\theta$. Then we have
\[
T^2(x,y)  = T \big(1+\theta-y ,x \big) = \big(\theta-y,1-x \big).
\]
Remark that we have $0<\theta-y<\theta$ and $1-\theta<1-x<1$. 

Now we compute $T^3(x,y)$ for some
 $(x,y)$ such that $0<x<\theta$ and $y>\theta$:
\[
T^3(x,y)  = T ^2\big(1+\theta-y ,x \big) = T\big(1+\theta-x,1+\theta-y \big) = \big(\theta-x,y-\theta \big).
\]

Thus we remark that a point inside $\{(x,y), y>\theta, x<\theta\}$ remains under the action of $T^3$ inside this set and its second coordinate decreases.
\begin{figure}[H]
\raisebox{-2cm}{
\begin{tikzpicture}[scale=1.1]
\fill[blue!30] (0,0.9) -- (0.6,0.9) -- (0.6,1.5) -- (0,1.5)  -- cycle;
\draw (0,0.9) -- (0.6,0.9) -- (0.6,1.5) -- (0,1.5)  -- cycle;
\draw (0,0) -- (2,0) -- (2,2) -- (0,2)  -- cycle;
\draw (2,0) -- (2.6,0) -- (2.6,2) -- (2,2)  -- cycle;
\draw (0.6,0)--(0.6,2);
\draw[->] (0.2,0.9) arc (0:90:0.2);
\node[below]  at (0,0) { $0$};\node[below] at (2,0) { $1$}; \node[below] at (2.6,0) { $1+\theta$};\node[below] at (0.6,0) { $\theta$};
\node[left]  at (0,0) { $0$};\node[left] at (0,2) { $1$}; \node[left] at (0,0.9) {$a$}; \node[left] at (0,0.6) {$\theta$};
\filldraw[color=lightgray,pattern=crosshatch] (0,0)--(0.6,0)--(0.6,0.6)--(0,0.6)--cycle ;
\end{tikzpicture}
}
$\underset{\epsilon=-1}{\xrightarrow[  \qquad T^3 \qquad ]{} }$
\raisebox{-2cm}{
\begin{tikzpicture}[scale=1.1]
\fill[blue!30] (0,0.3) -- (0.6,0.3) -- (0.6,0.9) -- (0,0.9)  -- cycle;
\draw (0,0.3) -- (0.6,0.3) -- (0.6,0.9) -- (0,0.9)  -- cycle;
\draw (0,0) -- (2,0) -- (2,2) -- (0,2)  -- cycle;
\draw (2,0) -- (2.6,0) -- (2.6,2) -- (2,2)  -- cycle;
\draw (0.6,0)--(0.6,2);
\draw[->] (0.4,0.3) arc (180:90:0.2);
\node[below]  at (0,0) { $0$};\node[below] at (2,0) { $1$}; \node[below] at (2.6,0) { $1+\theta$};\node[below] at (0.6,0) { $\theta$};
\node[left]  at (0,0) { $0$};\node[left] at (0,2) { $1$}; \node[left] at (0,0.3) {$a-\theta$}; \node[left] at (0,1.4) {$1-\theta$};
\filldraw[color=lightgray,pattern=crosshatch] (0,1.4)--(0.6,1.4)--(0.6,2)--(0,2)--cycle ;
\end{tikzpicture}
}
\end{figure}

We deduce  that
\begin{equation}\label{equ-renorm}\T_\omega(z) =\begin{cases}T^{3(n_\omega-1)} T^2(z)  & \text{if} \quad z\in \C_\omega,\\ 
T^3(z) & \text{if}\quad z \in  \R_\omega. \end{cases}
\end{equation} Then, if we denote $s(x)=\theta-x$ we obtain

\begin{equation}\label{equ-renorm1}
\T_\omega(x,y)=\begin{cases}   \big(s^{n_\omega-1}(\theta-y),1-x-\theta(n_{\omega}-1) \big) &  \text{if } \ z \in \C_\omega,\\   
\big(s^{n_\omega}(y),1-x-\theta(n_{\omega}-1) \big) &  \text{if } \ z \in  \R_\omega.
  \end{cases} 
  \end{equation}

To finish the proof we compute the conjugation of $\T_\omega$:

$$	\psi_{\omega}  \T_\omega \psi_\omega ^{-1} (x,y) =
 \begin{cases}  \left( 1-y+\frac{1}{\theta} - n_\omega ,f_{(-1)^{n_\omega+1}}(x) \right), \\ 
 (x-1,1-y)	.	\end{cases}$$
 
 Finally we have proven 
 $$\psi_{\omega}  \T_\omega \psi_\omega ^{-1} (x,y) =T_{S\omega}(x,y).$$
 
$\bullet$ Now we treat the case $\varepsilon=1$.  The computations are based on the same method, thus we just give the result:
$$
\T_\omega(z)=  T^{3(n_\omega-1)} T(z) \  \text{if} \ z\in\C_\omega
	\quad \text{and} \quad 
	\T_\omega(z)=T^3(z) \ \text{if} \ z\in  \R_\omega.
$$ 

Then, for $z=(x,y)$:
\begin{equation}\label{equ-renorm2}
\T_\omega(z)=\begin{cases}  \big(1+\theta-y - (1-\theta)(n_\omega-1), s^{n_\omega-1} (1-x)\big)  & \text{if} \  \ z \in\C_\omega,\\   
\big(\theta+x-1,1+\theta-y \big) & \text{if} \ \ z \in  \R_\omega.
\end{cases}
\end{equation}

Thus we obtain $T_{S\omega}(x,y)=\psi_{\omega}  \T_\omega \psi_\omega ^{-1} (x,y)$.

\begin{figure}[H]
\raisebox{-2cm}{
\begin{tikzpicture}[scale=1.1]
\fill[blue!30] (1.2,1.4) -- (1.8,1.4) -- (1.8,2) -- (1.2,2) -- cycle;
\draw (1.2,1.4) -- (1.8,1.4) -- (1.8,2) -- (1.2,2) -- cycle;
\draw (0,0) -- (2,0) -- (2,2) -- (0,2)  -- cycle;
\draw (2,0) -- (3.4,0) -- (3.4,2) -- (2,2)  -- cycle;
\draw (0,1.4)--(2,1.4);
\draw[dashed] (0.6,0)--(0.6,1.4);
\draw[dashed] (1.2,0)--(1.2,1.4);
\node[below]  at (0,0) { $0$};\node[below] at (2,0) { $1$}; \node[below] at (3.4,0) { $1+\theta$};
	\node[below] at (0.6,0) { $1-\theta$};\node[below] at (1.2,0) { $a$};
\node[left]  at (0,0) { $0$};\node[left] at (0,2) { $1$}; \node[left] at (0,1.4) { $1-\theta$}; 
\filldraw[color=lightgray,pattern=crosshatch] (0,1.4)--(0.6,1.4)--(0.6,2)--(0,2)--cycle ;
\draw[->] (1.4,1.4) arc (0:90:0.2);
\end{tikzpicture}
}
$\underset{\epsilon=1}{\xrightarrow[  T^3  ]{} }$
\raisebox{-2cm}{
\begin{tikzpicture}[scale=1.1]
\fill[blue!30] (0.6,1.4) -- (1.2,1.4) -- (1.2,2) -- (0.6,2) -- cycle;
\draw (0.6,1.4) -- (1.2,1.4) -- (1.2,2) -- (0.6,2) -- cycle;
\draw (0,0) -- (2,0) -- (2,2) -- (0,2)  -- cycle;
\draw (2,0) -- (3.4,0) -- (3.4,2) -- (2,2)  -- cycle;
\draw (0,1.4)--(2,1.4);
\draw[dashed] (0.6,0)--(0.6,1.4);
\draw[dashed] (1.4,0)--(1.4,1.4);
\node[below]  at (0,0) { $0$};\node[below] at (2,0) { $1$}; \node[below] at (3.4,0) { $1+\theta$};
	\node[below] at (1.4,0) { $\theta$};\node[below] at (0.6,0) { $a-\theta$};
\node[left]  at (0,0) { $0$};\node[left] at (0,2) { $1$}; \node[left] at (0,1.4) { $1-\theta$}; 
\filldraw[color=lightgray,pattern=crosshatch] (1.4,1.4)--(2,1.4)--(2,2)--(1.4,2)--cycle ;
\draw[->] (0.8,2) arc (360:270:0.2);
\end{tikzpicture}
}
\end{figure}

\subsection{Some important corollaries of Proposition \ref{induction}}
%%%%%%%%%%%%%%%%%%%%%%
%%%%%%%%%%%%%%%%%%%%%%

Let us define two substitutions
\begin{equation} \label{equation:subsitution}
\sigma_{(\theta,-1)} : 	\left\{		\begin{array}{ll}		a \to ab(aab)^{n_\omega-1} \\ 	b \to aab 	\end{array} 	\right. 
\quad \mbox{ and } \quad
\sigma_{(\theta,1)} :  	\left\{ 	\begin{array}{ll} 	a \to a(aab)^{n_\omega-1},  \\ 	b \to aab. 	\end{array} 	\right. 
\end{equation}

\begin{corollaire}[Language and substitution]\label{coro:mot-subst}
For $z \in \X_\omega = \C_\omega \cup \R_\omega$, we have 
\begin{center} $\Phi_\omega(z) = \sigma_\omega \circ  \Phi_{S(\omega)}  \circ \psi_\omega (z)$. \end{center}
\end{corollaire}
\begin{proof}
By Proposition \ref{induction}, we have $$\psi_\omega  \circ \T_\omega(z) =  T_{S(\omega)}\circ \psi_{\omega}(z).$$
If $z\in \mathcal C_\omega^{ind}$ we obtain by Equation \eqref{equ-renorm}
\[ \psi_\omega  \circ T^{3(n_\omega-1)} T^2(z) =  T_{S(\omega)}\circ \psi_{\omega}(z). \]
The infinite word $ \Phi_{S(\omega)}  \circ \psi_\omega (z)$ begins by $a$, and the proof of Proposition \ref{induction} shows that $\Phi_\omega(z)$  begins by $ab(ab)^{n_\omega-1}$. Since the same argument works if $z$ belongs to $\mathcal R_\omega^{ind}$ we deduce the result.

\end{proof}

These substitutions define  linear maps by abelianization.
These linear maps have matrices in $\mathcal M_2(\mathbb Z)$, called the incidence matrices and denoted  by:
\[
 M(\theta,-1) = \begin{pmatrix} 2n_{\omega}-1 &2 \\ n_{\omega}  &1 \end{pmatrix} 
 	\quad \mbox{ and }  \quad 
	M(\theta,1) = \begin{pmatrix} 2n_{\omega}-1 &2\\ n_{\omega}-1 &1 \end{pmatrix}.
\]

In Section \ref{section:cocycleover} we will return on the cocycle generated by the map $\omega \to M(\omega)$.

We denote these matrices and their coefficients as $M(\omega)=(m_{i,j}(\omega))_{1\leq i,j\leq 2}$.
The vector space $\mathbb R^2$ is equipped with norm $\| \cdot \|_1$. This defines a norm on $\mathcal M_2(\mathbb R)$ 
	by $\| \binom{a \ b}{c \ d} \|_1=\max\{|a|+|c|,|b|+|d|\}$.

\begin{definition} \label{defcocycle}
Let $\theta$ be an irrational number and $\omega=(\theta,\epsilon)$.
For each positive integer $i$ we denote $\omega_i = (\theta_i,\epsilon_i)=S^i(\theta,\epsilon)$ (we will see in Proposition \ref{Squadratique} that  $S^i(\omega)$ exists for each integer $i$).
Let us define also
\[
M^{(k)}(\omega) = M(\omega) \times \cdots \times M(\omega_k)  = 
	\begin{pmatrix} m_{1,1} ^{(k)}(\omega) & m_{1,2} ^{(k)}(\omega)  \\ m_{2,1} ^{(k)}(\omega) & m_{2,2} ^{(k)}(\omega)  \end{pmatrix}.
\]
Then we define a sequence of numbers $(p_i)_{i\in\mathbb N}$ by
\begin{equation} \label{eq:pi}
p_i = \left\|   M^{(i-1)}(\omega) \times k_i   \right\|_1,\quad 
\mbox{ where } k_i = \begin{pmatrix} 1 \\  \frac12(1+\epsilon_i) \end{pmatrix}.
	%\begin{cases}k_i = \begin{pmatrix} 1\\0 \end{pmatrix} \mbox{ if } \epsilon_i=-1 \\ k_i=  \begin{pmatrix} 1\\1 \end{pmatrix} \mbox{ if } \epsilon_i=1 
 	%\end{cases}
\end{equation}
\end{definition}

\begin{corollaire}[Periodic points]\label{coro:renormalisation}
We deduce
\begin{enumerate}
 \item Let $\mathcal I$ be a periodic cell for $T_{S(\omega)}$ associated to the periodic word $u$. 
Then $\psi_\omega^{-1}(\mathcal I)$ is a periodic cell for $T_\omega$ associated to the word $\sigma_\omega(u)$. Moreover its period is given by
\begin{equation} \label{eq:lienperiode}
	\left\|    M(\omega) \times \binom{ |u|_a }{ |u|_b }  \right\|_1\end{equation}
	where by definition $u$ contains $|u|_a$ letters $a$ and $|u|_b$ letters $b$.
\item Conversely if $\mathcal I$ is a periodic cell of period $n$ for $T_\omega$ ($n\geq 2$ for $T_{(\theta,-1)}$ and $n\geq 3$ for $T_{(\theta,1)}$), then there exists an integer  $0\leq k\leq n-1$ such that $\psi_\omega \big(T_\omega^{k} (\mathcal I) \big)$ is a periodic cell for $T_{S(\omega)}$.	
\item If $\mathcal I$ is a periodic cell, there exists an integer $n$ such that its period is $p_n$, and its area is equal to
\[   \prod \limits_{k=1}  ^{n} \frac{1}{ r(\omega_k)}.	\]
\item Each periodic cell is a square.
\item The set $\mathcal K_{(\theta,\epsilon)}$ is non empty if and only if $\theta \notin \mathbb Q$.
\end{enumerate}
\end{corollaire}
\begin{proof}
$\;$
\begin{enumerate}
 \item We use the relation $ \psi_\omega  \circ \T_\omega \circ \psi_{\omega}^{-1} =  T_{S(\omega)}.$
%It gives $\psi_\omega  \circ \T_\omega^k \circ \psi_{\omega}^{-1} = T_{S(\omega)}^k.$ Then assume 
If $m$ is a periodic point for $T_{S\omega}$ of period $k$, we obtain that $\psi_\omega^{-1}(m)$ is periodic for $\T_\omega$.
Now Equation \ref{equ-renorm} gives that $\T_\omega$ is some power of $T_\omega$. The periodic word associated to this cell is obtained with the first point of the Corollary \ref{coro:mot-subst}.

\item Consider a periodic cell $\mathcal{I}$ of $T_\omega$ of period $n$. We claim that there exists $k\leq n-1$ such that 
$T_\omega^k\mathcal{I}\in \C_\omega\cup \R_\omega$.
This fact is proven by remarking that 
$$
\begin{array}{lll}
\text{for}\quad \epsilon=-1 &: \quad &\big(\C_\omega\cup T_\omega\C_\omega\cup\dots\cup T_\omega^{3n_\omega-2}\C_\omega\big)\cup
  \big(\R_\omega\cup T_\omega\R_\omega\cup T_\omega^{2}\R_\omega\big),\\
\text{for}\quad\epsilon=1  &: \quad  &\big(\C_\omega\cup T_\omega\C_\omega\cup\dots\cup T_\omega^{3n_\omega-3}\C_\omega\big)\cup
  \big(\R_\omega\cup T_\omega\R_\omega\cup T_\omega^{2}\R_\omega\big),
\end{array}
$$
 cover $\mathcal X_\omega$ except the cells of period one ($\epsilon=-1$) or two ($\epsilon=1$).
By the proof of Proposition \ref{induction}, we know that these sets are disjoint. Then we compute the area:
\begin{itemize}
\item If $\epsilon=-1$ we obtain: 
$$
(3n-1)|\C_\omega|+3|\R_\omega|=(3n-1)\theta^2+3\theta(1-n\theta) = 3\theta-\theta^2.
$$
Now we compute the area of $\mathcal X_\theta\setminus \mathcal I_\omega(2)$. 
We obtain $1+\theta-(1-\theta)^2=3\theta-\theta^2$.
\item If $\epsilon=1$ we obtain: 
$$
(3n-2)|\C_\omega|+3|\R_\omega|=(3n-2)(1-\theta)^2+3(1-\theta)\big(1-n(1-\theta)\big) = (1-\theta)(1+2\theta).
$$
Now we compute the area of $\mathcal X_\theta\setminus \mathcal I_\omega(1)$. 
We obtain $1+\theta-2\theta^2$. 
\end{itemize}
Thus in all the cases we have shown that the complement of the cells of period at most two is equal to the first return sets of the induction zone.

Then by the claim consider the cell $T_\omega^k\mathcal I$. It is clearly a periodic cell for $\T_\omega$. By the first point of the corollary we deduce that their iterations by $\psi_\omega$ form a periodic orbit for $T_{S\omega}$. 

\item  Let $\mathcal I$ be a periodic cell of $T_\omega$ of period $p$. We apply the preceding result and deduce that for some $k\leq p-1$ the cell $\psi_\omega(T_\omega^k\mathcal I)$ is a periodic cell for $T_{S\omega}$. The period of this cell is strictly less than $p$ since the orbit of $\mathcal I$ under $T_\omega$ does not stay inside $\X_\omega$.  
We apply this argument recursively and at some step $n$ we obtain a cell of period $1$ or $2$. The first point of the Corollary allows us to deduce that $p=p_n$. The cell is thus given as the image of a square by the composition of the similitudes $\psi_\omega, \psi_{\omega_1}, \dots,  \psi_{\omega_n}$. The formula of the area follows.

\item By the previous point, the cell is the image by a similitude of a square.
\item We use Proposition \ref{Squadratique}. \qedhere
\end{enumerate} 
\end{proof}

\subsection{Aperiodic set} \label{subssystemesimilitudes}
%%%%%%%%%%%%%%%%%%%%%%%
%%%%%%%%%%%%%%%%%%%%%
Consider an irrational number $\theta$. Here we describe a partition of the aperiodic set which will be used in Section \ref{subsec:dynamic}.

\begin{lemme} \label{lemme:recouvrement}
For every integer $\ell$, the set $\mathcal K_\omega(p_\ell)$ has a partition (up to a set of zero measure) defined by $$\displaystyle\bigcup_{1\leq k\leq \left\|M^{(\ell)}  \binom 1 1 \right\|_1}\mathcal P_k^{(l)}$$ such that
\begin{itemize}
\item  each set $\mathcal P_k^{(l)}$ for $1\leq k\leq \left\|M^{(\ell)} \binom 1 0 \right\|_1$ is the image by a similitude of 
 $\mathcal C_{\omega_\ell}$. 
\item  Each set $\mathcal P_k^{(l)}$ for $\left\|M^{(\ell)} \binom 1 0 \right\|_1<k\leq \left\|M^{(\ell)} \binom 1 0 \right\|_1+\left\|M^{(\ell)} \binom 0 1  \right\|_1$ is the image by a similitude of  $\mathcal R_{\omega_\ell}$.
\item Each similitude has a ratio equal to $\displaystyle\prod_{k=1}^\ell \frac{1}{r(\omega_k)}$.
\end{itemize}

\end{lemme}
\begin{proof}
We fix $\omega=(\theta,\epsilon)\in\Omega$ with $\theta$ an irrational number and for each integer $\ell$, $S^\ell\omega = \omega_\ell = (\theta_\ell,\epsilon_\ell)$.
By definition $\mathcal K_\omega(p_1)$ is the closure of the complement of $\mathcal I_\omega(p_1)$. 
The proof of Proposition \ref{induction} shows that the orbit of $ \X_\omega$ is equal to $X\setminus \mathcal I_\omega(p_1)$. We deduce that $\mathcal K_\omega(p_1)$ is the closure of the orbit of $\X_\omega$. 
By equality \eqref{equ-renorm} we deduce 
$$
\begin{array}{ll}
\text{for}\quad\epsilon=-1 : \quad &\mathcal K_{\omega}(p_1) = \big(\C_\omega\cup T_\omega\C_\omega\cup\dots\cup T_\omega^{3n_\omega-2}\C_\omega\big)\cup
  \big(\R_\omega\cup T_\omega\R_\omega\cup T_\omega^{2}\R_\omega\big),\\
\text{for}\quad \epsilon=1: \quad  &\mathcal K_{\omega}(p_1) =  \big(\C_\omega\cup T_\omega\C_\omega\cup\dots\cup T_\omega^{3n_\omega-3}\C_\omega\big)\cup
  \big(\R_\omega\cup T_\omega\R_\omega\cup T_\omega^{2}\R_\omega\big).
\end{array}
$$
We recall  that $\C_{\omega} = \psi_\omega^{-1} (\mathcal C_{S \omega})$ and   $\R_{\omega} = \psi_\omega^{-1} (\mathcal R_{S \omega})$.
In each case, 
$$
\mathcal K_\omega (p_1) = 
 \bigcup \limits_{i=0}^{m_{1,1} (\omega) + m_{2,1} (\omega)-1} T^i \circ \psi_\omega ^{-1} \mathcal{C}
	\ \ \bigcup \  \ \bigcup \limits_{i=0}^{m_{1,2} (\omega) + m_{2,2} (\omega)-1} T^i \circ   \psi_\omega ^{-1} \mathcal{R}_{\theta_n}.
$$
The formula is proven for $\ell=1$.
We can repeat this to find the formula for each integer $\ell$ by using Corollary \ref{coro:renormalisation}. 

\begin{equation} \label{dsqkljhgshkdjlgflshk}
\mathcal K_\omega (p_\ell) = \bigcup \limits_{i=0}^{m_{1,1}^{(\ell)} (\omega) + m_{2,1}^{(\ell)}(\omega)-1} T^i \circ \psi_{\omega_\ell} ^{-1} \circ \cdots \circ \psi_\omega ^{-1} \mathcal{C}
	\ \cup \ \bigcup \limits_{i=0}^{m_{1,2}^{(\ell)} (\omega) + m_{2,2}^{(\ell)}(\omega)-1} T^i \circ \psi_{\omega_\ell} ^{-1} \circ \cdots \circ \psi_\omega ^{-1} \mathcal{R}_{\theta_\ell}.
\end{equation}
And the union is disjoint by Corollary \ref{coro:renormalisation}.
\end{proof}
%We could also give precise formulas for $\mathcal K_\omega ^c (p_n) $ and $\mathcal K_\omega^r (p_n) $ but they would be unusable.

\begin{remarque}
We prefer to work with the set $\mathcal K_\omega$ rather than $\mathcal D_\omega$, in particular because we 
should add a finite union of segments in Equation \eqref{dsqkljhgshkdjlgflshk}.
\end{remarque}

%%%%%%%%%%%%%%%%%%%%%%%%%%%%%%%%%%%%%%%%%%%%
%%%%%%%%%%%%%%%%%%%%%%%%%%%%%%%%%%%%%%%%%%%%
\section{The renormalization map} \label{se:autosimilaire}
%%%%%%%%%%%%%%%%%%%%%%%%%%%%%%%%%%%%%%%%%%%%
%%%%%%%%%%%%%%%%%%%%%%%%%%%%%%%%%%%%%%%%%%%%

The definition of the map $S$, in Section \eqref{def:S} Equation \ref{equ:S}, invites us to study the continued fraction expansion generated by this map. We will study the invariant measures for $S$ and will see that we should accelerate this map to get a nice dynamical system. We use an action by homography of $GL_2(\mathbb{R})$ on the real line $\mathbb{R}$ defined by:
$$\begin{pmatrix}a&b\\c&d\end{pmatrix}.x=\frac{ax+b}{cx+d}.$$

\subsection{Periodic points of $S$ and continued fractions} 
%%%%%%%%%%%%%%%%%%%%%%
%%%%%%%%%%%%%%%%%%%%%%

In order to do this we begin with the following remark:
The map $S$ is clearly non bijective, but it defines a continued fraction algorithm based on the fact that the equality 
$$\begin{cases}S(\theta, 1)=\left( \{\frac{1}{1-\theta}\},(-1)^{n+1} \right)\\
S(\theta,-1)=\left( \{\frac{1}{\theta}\},(-1)^{n+1} \right)\end{cases}$$ 
yields :
$$ \begin{cases}
\theta=1-\frac{1}{n+\theta_1}=\frac{n-1+\theta_1}{n+\theta_1}\\
\theta=\frac{1}{n+\theta_1}
\end{cases}
\quad \text{where} \quad
S(\theta,\epsilon) = (\theta_1,\epsilon_1)  \ \text{and} \ n = n_\theta.
$$

Thus we will speak about $S$-{\bf continued fraction expansion}. The sequence $(\theta_n,\epsilon_n)_{n\in\mathbb{N}}$ defined by 
$(\theta_n,\epsilon_n)=S^n(\theta,\epsilon)$ is called the {\bf $S$-expansion} of $(\theta,\epsilon)$.
A point $(\theta,\epsilon)$ is called {\bf an ultimately periodic point} for the continued fraction algorithm if 
there exists an integer $n$ such that $S^{n+m}(\theta,\epsilon)=S^n(\theta,\epsilon)$ for every integer $m$. 

\begin{example}
$$S^3\left(\frac{3}{8},-1\right)=S^2\left(\frac23,-1\right)=S\left(\frac12,1\right)=(0,-1).$$
$$S^3\left(\frac{\sqrt 2}{2},-1\right)=S^2\left(\sqrt 2 -1,1\right)=S\left(\frac{\sqrt 2}{2},1\right)=(\sqrt 2-1,1).$$

In the first case we say that the $S$-continued fraction expansion is finite, and in the second we have an ultimately periodic $S$-continued fraction expansion:
$$\sqrt 2=\frac{1}{\sqrt 2/2}=\frac{1}{\frac{1}{1+\sqrt2 -1}}=\frac{1}{\frac{1}{1+1-\frac{1}{1+\sqrt 2/2}}}.$$
\end{example}

\begin{proposition}\label{Squadratique}
Let $\omega=(\theta,\varepsilon)$ be an element of $\Omega$.
\begin{itemize}
\item The point $\omega\in \Omega$ has a finite $S$-expansion if and only if $ \theta \in \mathbb Q \cap [0,1]$.
\item The point $\omega\in \Omega$ has an ultimately periodic $S$-expansion if and only if $\theta$ is a quadratic number.
\end{itemize}
\end{proposition}

\subsection{Proof of Proposition \ref{Squadratique}} \label{subse:preuveultimperiodique}
%%%%%%%%%%%%%%%%%%%%%%%%%%

First of all remark that we have $\dfrac{1}{1-\begin{pmatrix}1&n-1\\1&n\end{pmatrix}.x}=x+n$.

$\bullet$ If $\theta$ is a rational number it is clear that its expansion is finite. 
Now consider $r\in \mathbb Q \cap ]0,1[$, and denote $S(r,\epsilon)=(r_1,\epsilon_1)$, then
$r_1$ is a rational number with a denominator strictly less than that of $r$. 
Thus if $(r,\epsilon)$ has an infinite $S$ expansion we obtain an infinite strictly decreasing sequence of integers, contradiction.

$\bullet$ Assume $(\theta,\epsilon)$ has an ultimately periodic $S$ expansion.
 
Our algorithm can be written:
$$\begin{cases}S(\theta, 1)=\left(\begin{pmatrix}n&1-n\\-1&1\end{pmatrix}.\theta,(-1)^{n+1}\right),\\
S(\theta,-1)=\left(\begin{pmatrix}-n&1\\1&0\end{pmatrix}. \theta,(-1)^{n+1}\right).\end{cases}$$

Now we remark that 
$$\begin{pmatrix}n&1-n\\-1&1\end{pmatrix}^{-1}=\begin{pmatrix}1&n-1\\1&n\end{pmatrix},  
\begin{pmatrix}-n&1\\1&0\end{pmatrix}^{-1}=\begin{pmatrix}0&1\\1&n\end{pmatrix}
.$$
$$(\theta,\epsilon)=\begin{cases}
\left(\begin{pmatrix}1&n-1\\1&n\end{pmatrix}.\theta_1,\epsilon_1\right),\\
\left(\begin{pmatrix}0&1\\1&n\end{pmatrix}.\theta_1,\epsilon_1\right).
\end{cases}$$
This formulation is better to obtain a left action of $GL_2(\mathbb{Z})$. 

By assumption $S^n(\theta,\epsilon)=S^{n+m}(\theta,\epsilon)$ thus there exists two integer matrices $M,N$ such that
$$MN.\theta=N\theta$$
We obtain a quadratic polynomial equation. Thus $\theta$ is a quadratic number.

$\bullet$ Assume that $\theta$ is a quadratic number. Now remark that 
$$\begin{pmatrix}1&n-1\\ 1&n\end{pmatrix}=\begin{pmatrix}0&1\\ 1&1\end{pmatrix}.
\begin{pmatrix}0&1\\ 1&n-1\end{pmatrix}.$$ 
Thus we can write
$$(\theta,\epsilon)=\begin{cases}
\left(\begin{pmatrix}0&1\\ 1&1\end{pmatrix}.
\begin{pmatrix}0&1\\ 1&n-1\end{pmatrix}.\theta_1,\epsilon_1\right),\\
\left(\begin{pmatrix}0&1\\1&n\end{pmatrix}.\theta_1,\epsilon_1\right).
\end{cases}$$
This means that our algorithm can be seen as the usual algorithm where we add the number $0$ as digit (if $n=1$). 
Consider the classical expansion of $\theta$. By Lagrange's theorem, $\theta$ has an ultimately periodic expansion for the classical continued fraction algorithm. If we add zero, we also have an ultimately periodic expansion.

\subsection{Dynamical properties of the renormalization map}\label{sec:S} 
%%%%%%%%%%%%%%%%%%%%%
%%%%%%%%%%%%%%%%%%%%%%%

First we define a bijection from $\Omega$ to $]0,2[$ by  $(\theta,\epsilon) \to x=\theta + \frac12(\epsilon+1)$.
This allows us to pass from the system $(\Omega,S)$ to the new system defined on $]0,2[$ and we keep the notation $S$ for simplicity.
\begin{remarque}\label{notation-rem}
In all what follows we will denote by $M_x$ or $M_\omega$ he same class of objects, depending on the previous bijection.
 \end{remarque}
On $]0,2[$, $S$ can be expressed as :
$$
S(x)= \begin{cases} \left\{\dfrac{1}{x} \right\}+\dfrac{(-1)^{n_x-1}+1}{2} \quad&  \text{if}\quad 0<x<1,\\ 
	\left\{\dfrac{1}{2-x} \right\} +\dfrac{(-1)^{n_x-1}+1}{2} \quad & \text{if}\quad 1\leq x<2. \end{cases}
$$

\begin{figure}[H]
\includegraphics[width=10cm]{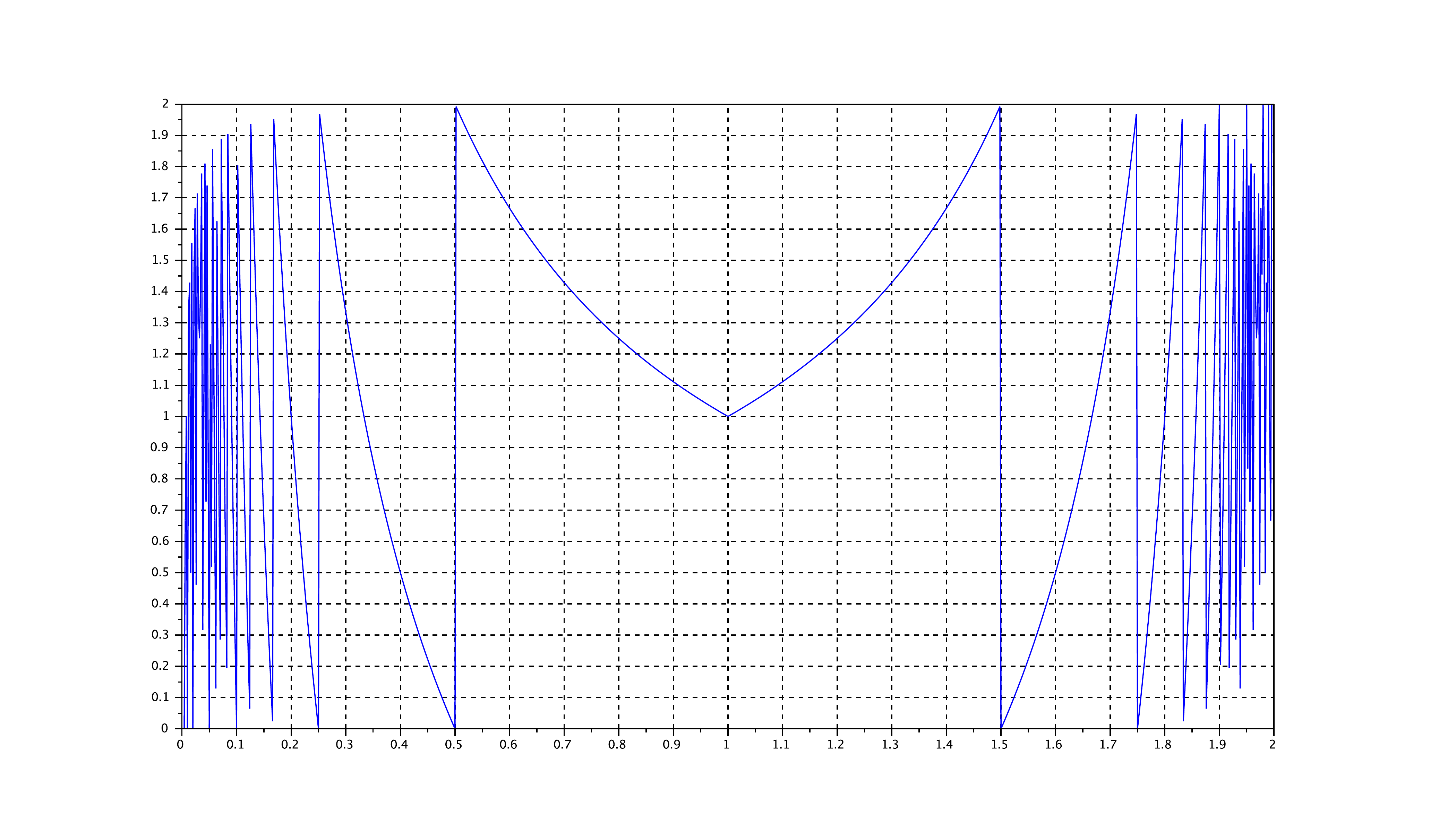}
\caption{Graph of the map $S$.}
\end{figure}

$S$ can also be expressed as $S(x)=A(x) \cdot x$ where $A(x)$ is defined for $n\geq 1$ by
\[
A(x)  = 
\begin{cases}
\begin{pmatrix} n'-n&1\\1&0 \end{pmatrix}  \quad& \text{if} \quad x\in \left] \dfrac1{n+1},\dfrac1{n}\right] , \\
 \begin{pmatrix} n-n'&1+2(n'-n)\\-1&2 \end{pmatrix} \quad& \text{ if} \quad x\in \left] 2-\dfrac1{n},2-\dfrac1{n+1}\right] , 
\end{cases}
\]
where $n'=n \mod 2$. Thus $S$ is a piecewise Moebius map, see Appendix \ref{subs:desciptionalgo}.

\begin{remarque}
In Proposition \ref{Sergodique}  we describe a method in order to determine an invariant measure for this map.
We obtain for the density function:

\[ \nu(x) =\begin{cases}  \dfrac1{x+1} & \mbox{ on } [0,1],  \\ \dfrac1{x(x-1)} &  \mbox{ on } ]1,2].\end{cases} \]

We do not develop this part further because this measure in not of finite volume and some important functions will not be $L^1$ integrable with respect to $\nu$. This  explains why we will consider the  accelerated dynamical system described in the next section.
\end{remarque}

\subsection{Acceleration of the renormalization map and ergodic properties}\label{sec:Sacc} 
%%%%%%%%%%%%%%%%%%%%%%
%%%%%%%%%%%%%%%%%%%%%%

The point $1$ is a  parabolic repulsive fixed-point for $S$, that is why we decide to accelerate the map in a neighborhood of this point, see \cite{Arn.Sch.13} for a complete reference.
\begin{remarque}
In order to simplify the notations, we will write in {\bf bold} all the objects which concern the accelerated map.
\end{remarque}

The acceleration of $S$ is denoted by $\S$ and is given by $ \S = S^{{\sf m}(x)}$ with
\[
{\sf m}(x) =
\begin{cases} 1 \quad  & \mbox{ if } \quad x\in[0,1]\cup[3/2,2],\\ 
\min\{n\in\mathbb N ; S^n x \notin [1,3/2] \} \quad & \text{otherwise}.
\end{cases}
\]
We can compute ${\sf m}$ and we get
\[
{\sf m}(x) =
\begin{cases} 1 \quad &  \mbox{ if } \quad x\in[0,1]\cup[3/2;2] , \\ 
k-1  \quad &\mbox{ if } \quad x \in \left[1+\frac1{k+1},1+\frac 1 k \right] \mbox{ with }k\geq 2.
\end{cases}
\]

\begin{figure}[H]
\includegraphics[width=10cm]{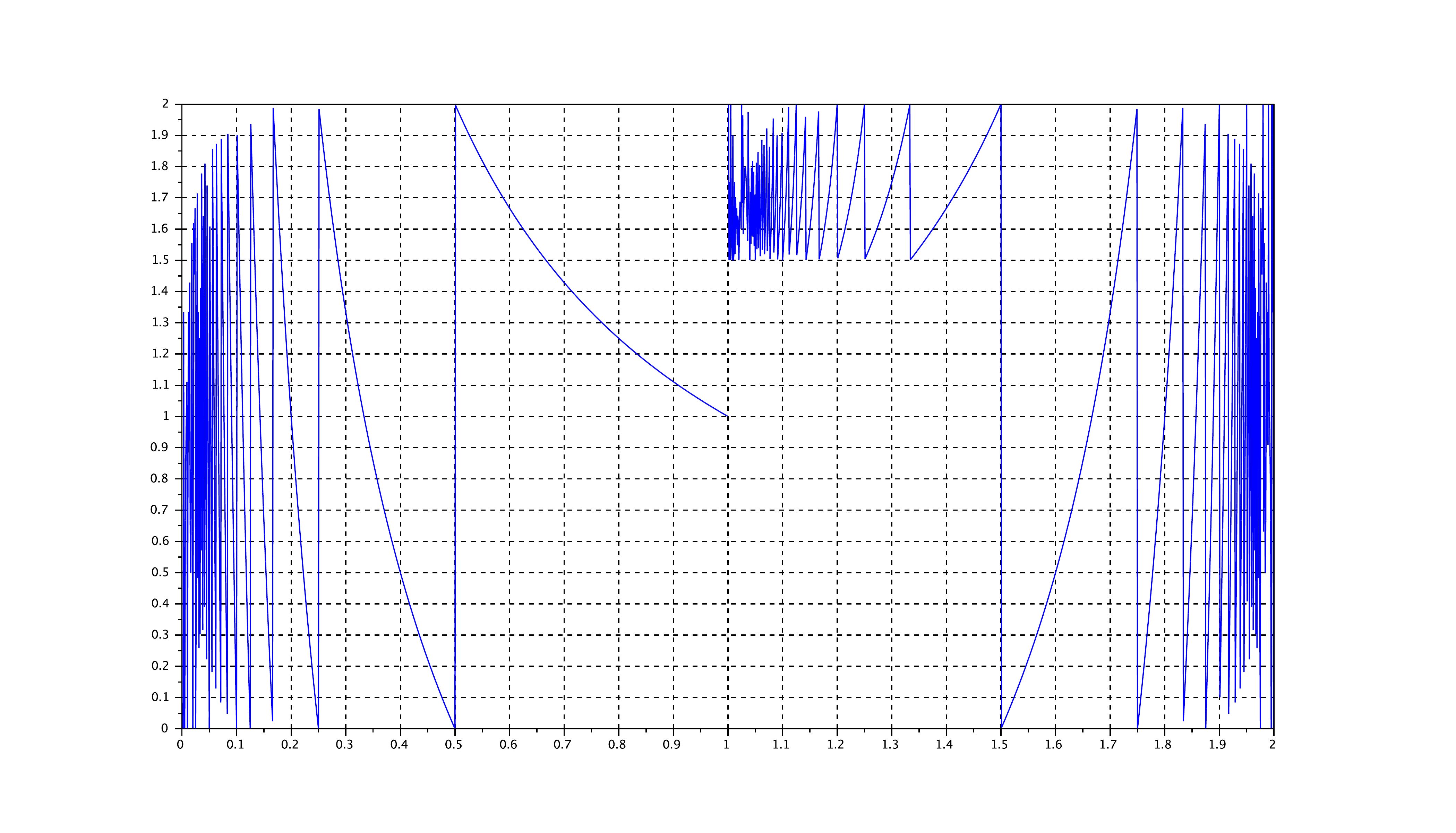}
\caption{Graph of the map $\S$.}
\end{figure}

We obtain $\S(x)= \A(x)\cdot x$ with:
\[
\A(x) = \begin{cases}
A(x) \quad &  \mbox{ if } \quad x\in[0,1]\cup[3/2;2] , \\ 
A(x)^{k-1}  = \begin{pmatrix} 2-k &k-1 \\1-k&k \end{pmatrix} \quad & \text{if} \quad x \in \left[1+\frac1{k+1},1+\frac 1 k \right] \mbox{ with }k\geq 2.
\end{cases}
\]

We will show in the Appendix that the method described in Proposition \ref{Sergodique} gives the following formula for the density of an invariant measure 

\[
\boldsymbol \nu(x)=\begin{cases} 
\dfrac{1}{1+x} & \mbox{on \quad $[0,1]$,}\\ \  \dfrac{1}{x} & \mbox{on \quad $[1,3/2]$, }\\ \dfrac{1}{x-1} & \mbox{on \quad  $[3/2,2]$.}\\
\end{cases}
\]

%\begin{remarque}
%It is a probability measure up to a factor $\boldsymbol \nu([0,2]) = \ln 6  $.
%\end{remarque}

\begin{remarque}
In Corollary 1 of \cite {Arn.Sch.13} the authors define the notion of map of {\bf first return type} and prove:
"If $T$ is of first return type and $\Gamma_\A$ is of finite covolume, then $T$ is ergodic with respect to the measure $\boldsymbol{\nu}$".
Unfortunately, the map $\A$ is not of first return type, thus we can not apply this result.
\end{remarque}
Thus we need to introduce another map in order to have some ergodic properties.

\subsection{Another map}
Consider the map defined on $]0,1[$ by
$$
Q(x)=\begin{cases}
\left\{\dfrac{1}{x} \right\} \quad x\in \left[\dfrac{1}{2n+1},\dfrac{1}{2n}\right], \quad  n\in\mathbb N^*, \\ 
1-\left\{\frac{1}{x}\right\} \quad x\in \left[\dfrac{1}{2n+2},\dfrac{1}{2n+1}\right], \quad n\in\mathbb N .
\end{cases}$$
Now let us define the map $p$:
$$\begin{array}{ccc}
[0,2]&\rightarrow&[0,1]\\
x&\mapsto&\begin{cases}x &\quad \text{if } \ x \in [0,1],\\ 2-x &\quad  \text{if} \ x\in[1,2].
\end{cases}
\end{array}$$

Remark that $S(1-x)=S(1+x)$ for any $x\in[0,1[$. Thus we have a commutative diagram:
\[
\begin{CD}
{[0,2[} @> {S} >> {[0,2[} \\
@ V{p} VV @ VV {p} V \\
{]0,1[} @>> {Q} > {]0,1[} \\
\end{CD}
\] 

\begin{figure}[H]
\includegraphics[width=10cm]{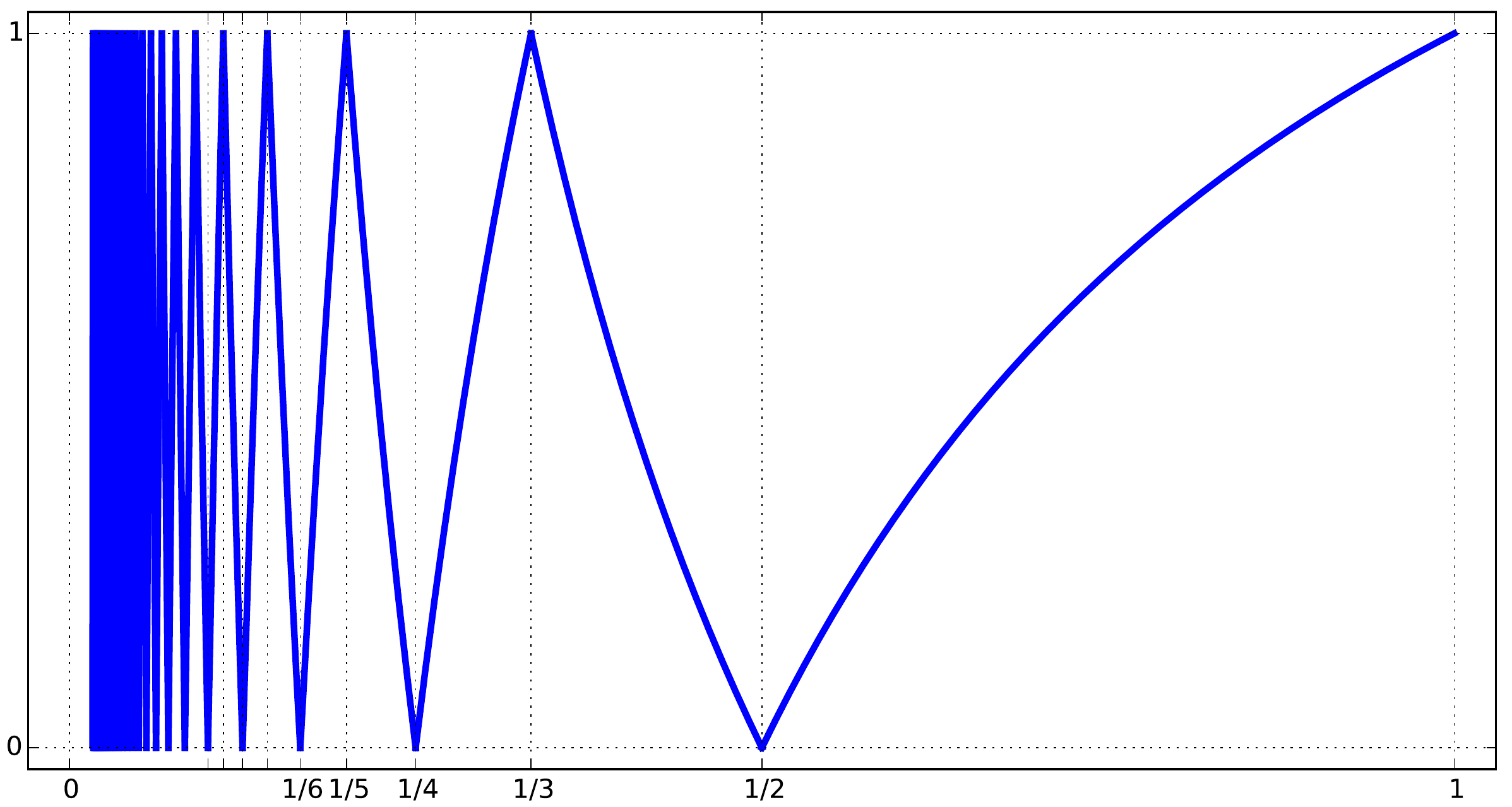}
\caption{Graph of the map $Q$.}
\end{figure}

We define the measure $\nu_1$ on $[0,1[$ by the formula $\nu_1(A)=\nu(p^{-1}(A))$. By definition this measure is $Q$-invariant.
We will use this application $Q$ to show the following result.
\begin{proposition}\label{prop:mes-inv-accS}
The dynamical system $(\Omega,\S,\boldsymbol\nu)$ is ergodic.
\end{proposition}

\subsection{Proof of Proposition \ref{prop:mes-inv-accS}} \label{subse:ergo}
%%%%%%%%%%%%%%%%%%%%%%%
%%%%%%%%%%%%%%%%%%%%%

The map $Q$ has the following properties
\begin{itemize}
\item It is defined on a countable union of intervals $I_n$, with value on an interval $I$.
\item On each interval $I_n$ the map $Q$ is a diffeomorphism.
\item The map $Q$ has {\bf bounded distortion}: there exists a constant $K>0$ such that
$$\sup_{I_n}\sup_{x,y\in I_n}\frac{|Q"(x)|}{|Q'(y)|^2}<K.$$
\end{itemize}

A classical result says that such a map is ergodic for the Lebesgue measure, see \cite{Luzz.} .

\begin{lemme}
We have:
\begin{itemize}
\item The map $Q$ is ergodic for the measure $\nu_1$ 
\item If $(I,Q,\nu_1)$ is ergodic, then $(\Omega, S ,\nu)$ is ergodic.
\item If $S$ is ergodic, then $(\Omega, \S ,\boldsymbol \nu)$ is ergodic. 
\end{itemize}
\end{lemme}
\begin{proof}
$\;$
\begin{itemize}
\item It is easy to see that the map $Q$ has the bounded distortion property.
Indeed, the classical Gauss map $G :x \mapsto  \left\{\frac1 x \right\}$ has the bounded distortion property and for $x \in [0,1]$, $|G'(x)| = |Q'(x)|$. Thus it is ergodic with respect to the Lebesgue measure. 
Now the measure $\nu$ is absolutely continuous with respect to the Lebesgue measure, see Subsection \ref{sec:S}. Thus the system $(I,Q)$ is ergodic for some measure. This measure is in the same class as the Lebesgue measure, thus the system $(I,Q,\nu_1)$ is ergodic.

\item Assume by contradiction that $S$ is not ergodic. Then there exists a set $A$ with $\nu(A)>0$ such that $S^{-1}A=A$. 
By symmetry, the set $A$ is symmetric with respect to $x=1$. Then there exists a set $A'\in [0,1)$ such that $Q^{-1}A'=A'$. 
By definition we have $\nu_1(A')=\frac{\nu(A)}{2}>0$. Contradiction.

\item The last part is a classical result.
\end{itemize}

\end{proof}

\subsection{Acceleration of the renormalization map as first return}\label{sec:SaccFR} 
%%%%%%%%%%%%%%%%%%%%%%%
%%%%%%%%%%%%%%%%%%%%%%

%The map $\S$ can be seen as temps de premier retour dans certaines zones.

\begin{definition} \label{defcocycleacc}
We consider the following substitutions $\boldsymbol \sigma_x = \sigma_x \circ  \cdots \circ  \sigma_{S^{{\sf m} (x)} x} $ associated to the matrices $ \M(x) =  M(x )\times \cdots \times  M(S^{{\sf m} (x)} x)$. We recall that $x\in ]0,2[$ is in bijection with $\omega\in \Omega$, thus $\sigma_x$ denotes the same object as $\sigma_\omega$, see Equation \ref{equation:subsitution}.

\begin{itemize}
\item  for $n\geq 1$, on  $\left[ \frac{1}{n+1}, \frac1n\right]$, ${\sf m} (x)  =1$ and  
	\[ 
	\boldsymbol \sigma_{x} : 	\left\{		\begin{array}{ll}		a \to ab(aab)^{n-1} \\ 	b \to aab 	\end{array} 	\right.
	\quad \mbox{ and } \quad 
	\M(x) = \begin{pmatrix} 2n-1&2\\n&1 \end{pmatrix},
	\]
\item for $n\geq 2$, on  $\left[1+ \frac{1}{n+1},1+ \frac1n\right]$, ${\sf m} (x)  =n-1$ and
	\[ 
	\boldsymbol \sigma_{x} : 	\left\{		\begin{array}{ll}		a \to a \\ 	b \to a^{2(n-1)} b 	\end{array} 	\right.
	\quad \mbox{ and } \quad 
	\M(x) = \begin{pmatrix} 1&2(n-1)\\0&1 \end{pmatrix},
	\]
\item for $n\geq 2$, on  $\left[2- \frac{1}{n},2- \frac1{n+1}\right]$, ${\sf m} (x)  =1$ and
	\[ 
	\boldsymbol \sigma_{x} :  	\left\{ 	\begin{array}{ll} 	a \to a(aab)^{n-1}  \\ 	b \to aab 	\end{array} 	\right. 	
	\quad \mbox{ and } \quad 
	\M(x) = \begin{pmatrix} 2n-1&2\\n-1&1 \end{pmatrix}.
	\]
\end{itemize}

Let us define also
\[
\M^{(k)}(x) = \M(x) \times \cdots \times \M(\S^k x)  = 
	\begin{pmatrix} \m_{1,1} ^{(k)}(x) & \m_{1,2} ^{(k)}(x)  \\ \m_{2,1} ^{(k)}(x) & \m_{2,2} ^{(k)}(x)  \end{pmatrix}.
\]

Now let us define the normalization factors with the help of Subsection \ref{def:S}.
\[
\r(x) =  r(x)  \times \cdots \times r(S^{{\sf m} (x)} x) .
\]

A simple calculation gives
\[
\r(x) = \begin{cases}
		\dfrac 1x &\quad \mbox{ if } x\in]0,1], \\
		 \dfrac 1 {n-(n-1)x}  &\quad \mbox{ if } x \in[1,3/2[ \mbox{ where } n = \floor{\dfrac 1 {x-1}} , \\ 
		\dfrac 1 {2-x} &\quad \mbox{ if } x\in]3/2,2] .
	\end{cases}
\]
\end{definition}

%%%%%%%%%%%%%%%%%%%%%%%%%%%%%%%%%%%%%%%%%%%%
%%%%%%%%%%%%%%%%%%%%%%%%%%%%%%%%%%%%%%%%%%%%
\section{Dynamics on the aperiodic set}\label{subsec:dynamic} 
%%%%%%%%%%%%%%%%%%%%%%%%%%%%%%%%%%%%%%%%%%%%
%%%%%%%%%%%%%%%%%%%%%%%%%%%%%%%%%%%%%%%%%%%%

\subsection{Background on Sturmian substitutions}

\begin{definition}[\cite{Pyth.02}]
An infinite word $u=u_0\cdots u_n\cdots $ over the alphabet $\{a,b\}$  is a {\bf Sturmian word} if one of the following conditions holds :
\begin{enumerate}
\item there exist an irrational number $\alpha\in [0,1]$ called {\bf the angle} of $u$ and $\beta\in \mathbb R$ such that
$$
\forall n\in \mathbb N, \ u_n=a \Leftrightarrow  n\alpha+\beta - \lfloor n\alpha+\beta\rfloor \leq 1
\mbox{ or }
\forall n\in \mathbb N, \ u_n=a \Leftrightarrow  n\alpha+\beta - \lceil n\alpha+\beta\rceil \leq 1,
$$
where $\lfloor \cdot \rfloor$ and $\lceil \cdot \rceil$ are respectively the floor and the ceiling functions,
\item the symbolical dynamical system associated to $u$ is measurably conjugated to a rotation on the circle by an irrational number.
\item for each integer $n$, card$\big(L_n(u)\big)=n+1$. 
\end{enumerate}
A substitution $\sigma$ is say to be {\bf sturmian} if the image of every sturmian word by $\sigma$ is a sturmian word.
\end{definition}

\subsection{One technical lemma}
%%%%%%%%%%%%%%%%%%%%%
%%%%%%%%%%%%%%%%%%%%%%%
We have the following result.
\begin{lemme}\label{prop:sturmian-subst}
$\;$
\begin{enumerate}
\item $\sigma$ is a sturmian substitution if and only if  it is a composition of the basics  substitutions :
$$
s_1:\left\{ \begin{array}{l} a\to ab \\ b\to b\end{array}\right. , \
s_2:\left\{ \begin{array}{l} a\to ba \\ b\to b\end{array}\right. , \
s_3:\left\{ \begin{array}{l} a\to a \\ b\to ba \end{array}\right.,\
s_4:\left\{ \begin{array}{l} a\to a \\ b\to ab\end{array}\right. \mbox{ and }
s_5:\left\{ \begin{array}{l} a\to b \\ b\to a\end{array}\right.
$$
\item If $(\sigma_i)_{i\in\mathbb N}$ is a sequence of sturmian substitutions, such that 
	$\sigma_1 \circ  \cdots \circ \sigma_{\ell} (a)$ converge to an infinite word $u$, 
	then $u$ is a sturmian word if and only if the sequence $\sigma_i$ is ultimately constant equal to some $s_i$ for $i\in\{1,2,3,4\}$.
\end{enumerate}
\end{lemme}

\begin{proof}
The first point is a consequence of \cite{Pyth.02}. Let us prove the second point:
we fix an integer $n$ and a word $u=\lim_{\ell \to +\infty} \sigma_1 \circ  \cdots \circ \sigma_{\ell} (a)$. We want to count the number of words of length $n$ factors of $u$. 
First by minimality we can find an integer $N$ such that every word of length $n$ appears in $u_1\dots u_N$. Then there exists an integer $m$ such that $u_1\dots u_N$ is factor of $\sigma_1 \circ  \cdots \circ \sigma_{m} (a)$ by definition of $u$. Now consider a sturmian word which begins by $a$. The word $\sigma_i(a\dots )$ is sturmian by definition of $\sigma_i$ for every integer $i\leq \ell$. We deduce that the number of factors of length $n$ in $\sigma_1 \circ  \cdots \circ \sigma_{m} (a)$ is bounded by $n+1$.

Now we use the fact that the basics substitutions are join to classical Gauss continued fractions.
Then the word is periodic if and only if the frequency of letters $a$ and $b$ are rational if and only if the expansion in continued fraction is finite. 
\end{proof}

\subsection{Result}
%%%%%%%%%%%%%%%%%%%%%
%%%%%%%%%%%%%%%%%%%%%%%

We proved in Proposition \ref{Squadratique}  that $S^n (\theta,\epsilon)$ is well defined for each integer $n$ if and only if $\theta$ is an irrational number. 
We fix for all this section an irrational number $\theta$ in $[0,1]$, $\epsilon\in\{-1,1\}$ and $\omega = (\theta,\epsilon)$.
By definition of the substitutions $\sigma_x$, we remark that $a$ is a prefix of $\sigma_x(a)$ for each substitution. This means that the following word is well defined 
$$
u_\omega = \lim_{\ell \to +\infty} \boldsymbol\sigma_\omega \circ  \cdots \circ \boldsymbol\sigma_{S^\ell \omega} (a).
$$
It is clearly an infinite word, because for each $\omega \in ]0,1[\times\{-1\}\cup ]1/2,1[\times \{1\}$ we have $|\boldsymbol\sigma_\omega(u)|>|u|$
and $]0,1/2[\times\{1\}$ is not stable by $S$.  Now we define $\Sigma_\omega:=\overline{\{t^nu_\omega, n\in\mathbb N\}}$, where $t$ is the shift map.

We can now state the main result of this section :
\begin{proposition}\label{prop:rot-aperiodique}
The dynamical system $(\mathcal K_\omega, T_\omega)$ is conjugate to an irrational rotation of the circle $S^1$.
\end{proposition}
The proof is a consequence of Lemma \ref{prop:sturmian-subst} and the two following lemmas.

\begin{lemme} \label{lemme:omegast}
The sequence $u_\omega$ is a Sturmian sequence.
\end{lemme}
\begin{proof}

With the previous Lemma, we only have to verify that for each $\omega=(\theta,\epsilon)\in \Omega$ such that $\theta$ is irrational, 
$\sigma_\omega$ is a sturmian substitution. 

Let $n$ be an integer. We define 
$$
p : \left\{ \begin{array}{l} a\to ab \\ b\to aab\end{array}\right., \quad
q : \left\{ \begin{array}{l} a\to a \\ b\to aab\end{array}\right. 
	\quad \mbox{and} \quad
r_n : \left\{ \begin{array}{l} a\to ab^n \\ b\to b \end{array}\right.
$$
It is clear that they are the composition of basics sturmian substitutions and that each substitution $\sigma_\omega$ defined in \eqref{equation:subsitution} is of the form :
\[
p\circ r_n \ : \ \left\{		\begin{array}{ll}		a \to ab(aab)^{n} \\ 	b \to aab 	\end{array} 	\right. 
\quad \mbox{ and} \quad 
q\circ r_n \ : \  \left\{ 	\begin{array}{ll} 	a \to a(aab)^{n}  \\ 	b \to aab .	\end{array} 	\right. 
\qedhere
\]
\end{proof}

\begin{remarque}
The relation between $\omega$ and the angle of the word $u_\omega$ is not obvious as see on the next Figure. We refer to Theorem \ref{conv-prod-matr} in Section \ref{subsecasparticulier}. The angle of $u_\omega$ is the first
coordinate of the vector $z_\omega$ where the sum of coordinates  is equal to $1$.
It is also equal to the frequency of one letter in the word $u_\omega$. Thus we can express it as $\lim \limits_{\ell} \frac{\left| (1,0)\times \M^{(\ell)}(\omega) \binom10 \right| }{\left| (1,1)\times \M^{(\ell)}(\omega) \binom10 \right|}$.
\end{remarque}

\begin{figure}[H]
\includegraphics[width=0.4\textwidth]{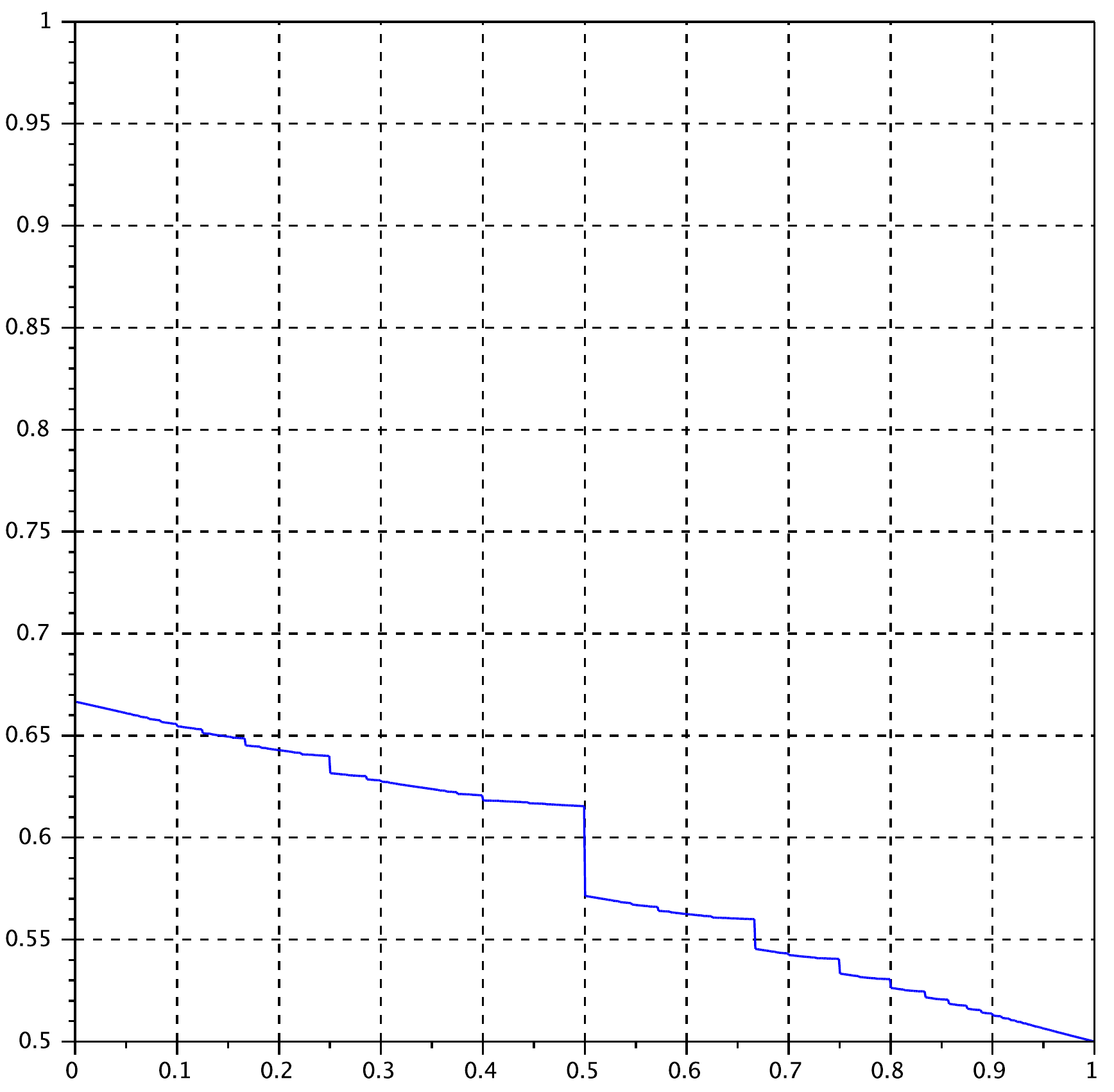}
\includegraphics[width=0.4\textwidth]{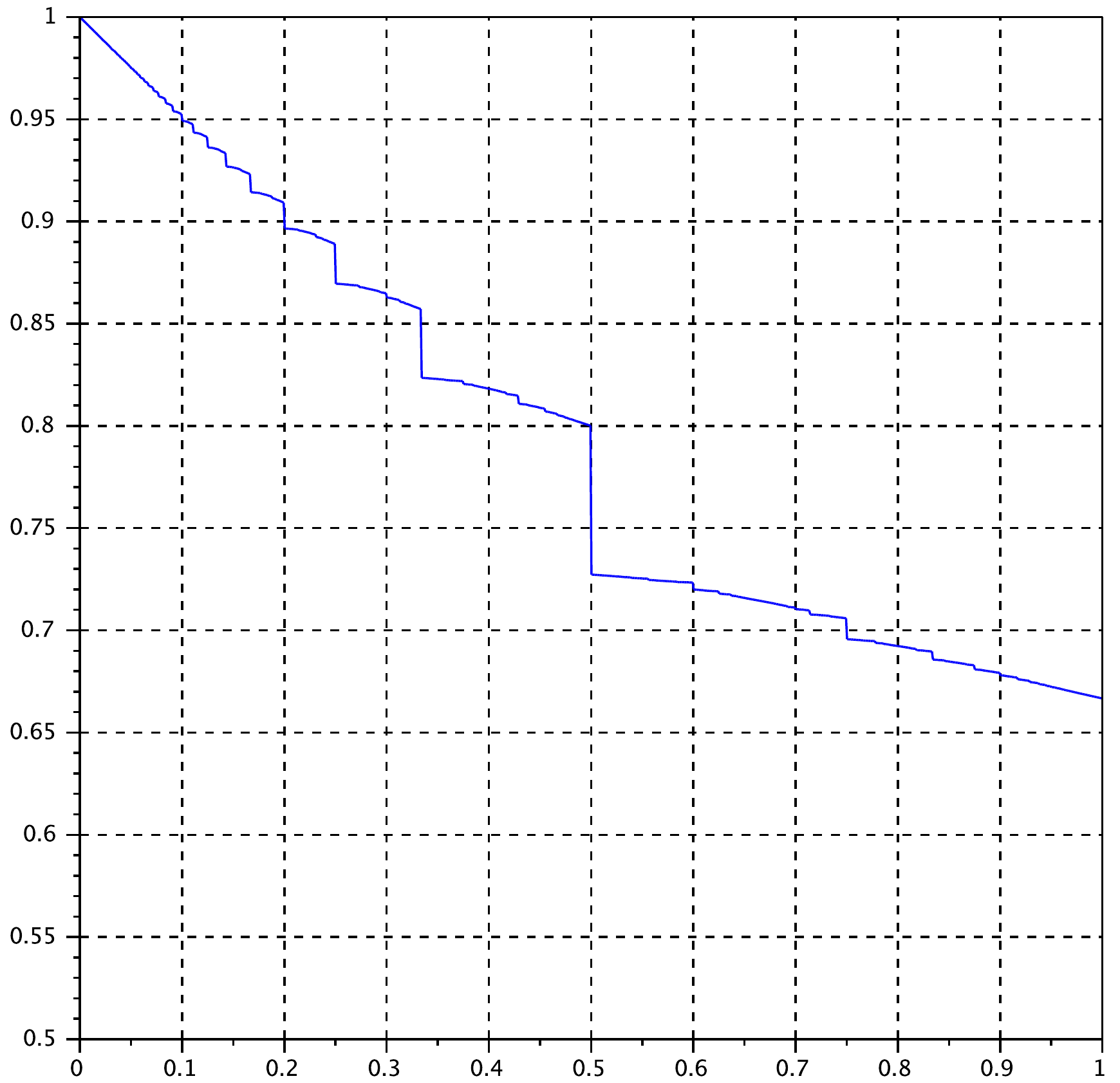}
\caption{Numerical representation of the map $\omega \to $ angle$(u_\omega)$ respectively for $\epsilon=-1$ and $\epsilon=1$.
 }

%\maltese c'est une forme bizarre mais je l'ai modifie et c'est bon.
\end{figure}

\begin{lemme} \label{lemme:systdynsymb}
The dynamical system $(\mathcal K_\omega,T_\omega)$ is conjugated to $(\Sigma_\omega, t)$.\end{lemme}
\begin{proof}
We will construct an explicit conjugation in Subsection \ref{subse:corresp}.
\end{proof}

We also denote the following quantities:
\begin{equation} \label{eq:NaNb}
\begin{array}{c}
\boldsymbol N_a^{(\ell)}=\boldsymbol N_a^{(\ell)}(\omega) = \left\|  \boldsymbol M^{(\ell)}(\omega) \binom10 \right\|_1, \quad
\boldsymbol N_b^{(\ell)}= \boldsymbol N_b^{(\ell)}(\omega) = \left\|  \boldsymbol M^{(\ell)}(\omega) \binom01 \right\|_1 
	 \\
 \quad \mbox{ and } \quad \boldsymbol N^{(\ell)}= \boldsymbol N^{(\ell)}(\omega) =\left\|  \boldsymbol M^{(\ell)}(\omega) \binom11 \right\|_1  .
\end{array}
\end{equation}

By Lemma \ref{lemme:omegast} the system $(\Sigma_\omega,t)$ is uniquely ergodic,
we denote $\tilde \mu _\omega$ the unique invariant probability measure of $(\Sigma_\omega,t)$.  

\begin{lemme}\label{lem:partitionl}
For every positive integer $\ell$ there exists a partition $\mathcal P_\ell$ of $\Sigma_{\omega}$ by the following sets such that:
$$\mathcal P_\ell=\{\boldsymbol  P^{(\ell)}_1,\dots, \boldsymbol P^{(\ell)}_{\N^{(\ell)}_a(\omega)}, \boldsymbol P^{(\ell)}_{\N^{(\ell)}_a(\omega)+1},\dots, \boldsymbol P^{(\ell)}_{\N^{(\ell)}(\omega)} \}.$$
\begin{itemize}
\item The partition $\mathcal P_{\ell+1}$ is a refinement of the partition $\mathcal P_{\ell}$. 
%$$\Sigma_{\omega}=\bigcup_{1\leq n\leq \N^{(\ell)}_a(\omega)}P^{(\ell)}_n\cup \bigcup_{1\leq n\leq \N^{(\ell)}_b(\omega)}Q^{(\ell)}_n$$
\item The function $n\mapsto \tilde\mu_\omega(\boldsymbol P^{(\ell)}_n)$ is constant on the intervals 
	$\{1,\ldots,\N^{(\ell)}_a\}$ and $\{\N^{(\ell)}_a+1,\ldots,\N^{(\ell)}\}$. 
\end{itemize}
\end{lemme}
\begin{proof}
Let $v=v_0\ldots v_n\ldots $ be an element of $\Sigma_\omega$ and let $\ell$ be an integer.

We denote by
$
 \boldsymbol \sigma_\omega \circ \cdots \circ \boldsymbol \sigma_{\S^\ell \omega}  = \boldsymbol \sigma^{(\ell)}_\omega.
$

By unique desubsitution, there exists an integer $n$ and an unique word $t_0\dots t_{n-1}$ such that
\begin{equation} \label{eq:decompositiondev}
v = t_0\ldots t_{n-1}  \boldsymbol \sigma^{(\ell)}_\omega (v_{\ell})
\mbox{ where } v_\ell \in \Sigma_{\S^\ell \omega}
\end{equation}
and $ t_0\ldots t_{n-1} $ is a suffix of $ \boldsymbol \sigma^{(\ell)}_\omega(a)$
	or $ \boldsymbol \sigma^{(\ell)}_\omega(b)$.

The word $ \boldsymbol \sigma^{(\ell)}_\omega(a)$ has length equal to $\N^{(\ell)}_a (\omega)$. Thus its number of suffixes is equal to $\N^{(\ell)}_a (\omega)$.  
Let $n\in\{1,\ldots,\N^{(\ell)}_a(\omega)\}$, we denote by $\boldsymbol P^{(\ell)}_n$ the set of words of the form \eqref{eq:decompositiondev}
where $t_0\cdots t_{n-1}$ is the suffix of $ \boldsymbol \sigma^{(\ell)}_\omega(a)$ of length $n$.
%$$P^{(\ell)}_n=[u], \quad u< Suff_n(\sigma^{(\ell)}(a)).$$
Similarly we define for $n\in\{1+\N^{(\ell)}_a(\omega),\N^{(\ell)}(\omega) \}$ $\boldsymbol P^{(\ell)}_n$ as the set of words of the form
\eqref{eq:decompositiondev}
where $t_0\cdots t_{n-1}$ is the suffix of $ \boldsymbol \sigma^{(\ell)}_\omega(b)$ of length $n$.

These sets are disjoint up to a measure zero set. Indeed the result is clear for two sets $P_i, P_j$ with $i,j\leq \N^{(\ell)}_a(\omega)$. 
If $v$ belongs to the intersection of some $\boldsymbol P^{(\ell)}_{\N^{(\ell)}_a(\omega)+k}, \boldsymbol P^{(\ell)}_{i}$ it means that the infinite word $v_\ell$ can be extended to the left by two letters $a,b$. The set of these words is of zero measure.

Thus we deduce:
$$\Sigma_{\omega}=\bigcup_{1\leq n\leq \N^{(\ell)}_a(\omega)}\boldsymbol P^{(\ell)}_n\cup \bigcup_{1\leq k\leq \N^{(\ell)}_b(\omega)}\boldsymbol P^{(\ell)}_{\N^{(\ell)}_a(\omega)+k}$$
Remark that $\boldsymbol \sigma_\omega^{(\ell+1)}(a)=\boldsymbol \sigma_\omega^{(\ell)}(\boldsymbol \sigma_{S^{\ell}\omega}(a))$. Thus in the sequences of such partitions, each partition refines the previous one.
Moreover, notice that the image of $\boldsymbol P^{(\ell)}_n$ by the shift map is equal to $P^{(\ell)}_{n-1}$ (analogously for $Q^{(\ell)}_n$). 
Since the measure $\tilde \mu_\omega$ is shift invariant, the result follows.
\end{proof}

\begin{definition}\label{def=mesure}
Let us denote these measures by $\alpha_\ell$ and $\beta_\ell$. Remark that $\N^{(\ell)}_a \alpha_\ell + \N^{(\ell)}_b \beta_\ell =1$.
\end{definition}
\subsection{Correspondance between $(T_\omega,\mathcal K_\omega)$ and the symbolic dynamical system}
	\label{subse:corresp}
%%%%%%%%%%%%%%%%%%%%%%%%%%%%%%%%%%%%%%%%%%%%
Consider an irrational number $\omega$. 
First we define a metric $\bar{d}$ on $\Sigma_\omega$. Two infinite words $u,v$ fulfill $\bar{d}(u,v)=\boldsymbol{R}^{-i}$, if $i=\max\{k\in\mathbb N, \exists n\in \mathbb N, u,v\in P_n^{(k)}\}$. We leave it to the reader to check, with the help of Lemma \ref{lem:partitionl}, that $\bar{d}$ is a metric, compatible with the topology on $\Sigma_\omega$. Remark that the ball of center $u$ and radius $\boldsymbol{R}^i$ is equal  to $P_n^{(i)}$ for some integer $n$.

We prove Lemma \ref{lemme:systdynsymb} in the following form.
\begin{proposition}\label{def:h}
There exists a map $h:\Sigma_\omega\rightarrow\mathcal K_\omega$ which is
\begin{itemize}
\item continuous,
\item almost everywhere injective and onto.
\item We have $T_\omega \circ h=h\circ t$.
\end{itemize}
\end{proposition}
	
\begin{proof}	
We use the preceding lemma and Equation \eqref{eq:decompositiondev}.

Let $\ell$ be an integer, and $v\in \Sigma_\omega$, then we define the set $B_\ell$ by
\[
B_\ell(v) = 
\left\{
\begin{array}{ll}
T^{n}  \circ \psi_{\omega}^{-1} \circ \cdots \circ \psi_{\S^{\ell}  \omega}^{-1} (\mathcal C)
	& \mbox{ if } \  \boldsymbol \sigma^{(\ell)}_\omega (v_{\ell})=a\dots \\
T^{n}  \circ \psi_{\omega}^{-1} \circ \cdots \circ \psi_{\S^{\ell}  \omega}^{-1} (\mathcal R)
	& \mbox{ if } \ \boldsymbol \sigma^{(\ell)}_\omega (v_{\ell})=b\dots  \\
\end{array}\right.
\]

By Proposition \ref{induction}, the restriction of the map $T^{n}$ to the set $\psi_{\omega}^{-1} \circ \cdots \circ \psi_{\S^{\ell}  \omega}^{-1} (\mathcal C)$ is an isometry. The same property is true for the restriction to the set $ \psi_{\omega}^{-1} \circ \cdots \circ \psi_{\S^{\ell}  \omega}^{-1} (\mathcal R)$. 
Each map $\psi_{\S^k \omega}$ is a similitude of ratio $\r(\S^k \omega)$. We deduce that each set $B_\ell(v)$  is included in a square of size $\boldsymbol{R}^{(\ell)} = \frac{1}{\r(\omega)} \cdots \frac{1}{\r(\omega_\ell)}$. 

By Proposition \ref{Squadratique}, the sequence $(\omega_\ell)_{\ell\in\mathbb N}$ is not periodic. By Lemma \ref{lem:afairebis} and Equation \eqref{lem:integration} we deduce that $\lim_{+\infty}\boldsymbol{R}^{(\ell)}=0$ for almost all $\omega$.
The sequence $(B_\ell(v))_\mathbb N$ is thus a decreasing sequence of compact sets whose diameters converge to zero, and therefore there is an unique element in the intersection. We denote it by
$$h(v)=\bigcap_{\ell} B_\ell(v).$$

$\bullet$ We claim that $h$ is a continuous function:
consider $\eta>0$ and $u,v$ infinite words in $\Sigma_\omega$ such that $\bar{d}(u,v)\leq \eta$. By the preceding Lemma this means that there exists an integer $\ell$ such that $u$ and $v$ are in the same atom  $\mathcal P_i$ of the partition for $i\leq \ell$. 
Then $h(u)$ and $h(v)$ belong to the same ball $B_\ell$. Thus the distance between $h(u)$ and $h(v)$ is bounded by the diameter of $B_\ell(u)$ which is bounded by $\boldsymbol{R}^{(\ell)}$. Since this number decreases to zero when $\ell$ goes to infinity, we deduce the continuity of $h$.

$\bullet$ The map is onto:  remark that Lemma \ref{lemme:recouvrement} implies that 
$$\mathcal K_\omega(p_\ell)=\bigcup_{ n\leq N^{(l)}}\bigcup_{v\in P_n^{(l)}}B_l(v).$$
Since $\mathcal K_\omega$ is obtained as the intersection of these sets, the surjectivity of the map is clear.

The injectivity comes from the definition, and the fact that there is only one point which has a prescribed symbolic aperiodic coding.

$\bullet$ The relation between the maps $h,T_\omega$ and $t$ is clear from the definition of $h$.
\end{proof}

\begin{definition}\label{def:mesure}
As a by-product of the result, this map defines a measure $\mu$ on $\mathcal K_\omega$ by the formula $\mu(A) = \tilde \mu_\omega\big(h^{-1}(A)\big)$.
\end{definition}

%%%%%%%%%%%%%%%%%%%%%%%%%%%%%%%%%%%%%%%%%%%%
%%%%%%%%%%%%%%%%%%%%%%%%%%%%%%%%%%%%%%%%%%%%
\section{Hausdorff dimension of the aperiodic set}\label{sec:calc-hausdorff} 
%%%%%%%%%%%%%%%%%%%%%%%%%%%%%%%%%%%%%%%%%%%%
%%%%%%%%%%%%%%%%%%%%%%%%%%%%%%%%%%%%%%%%%%%%

%%%%%%%%%%%%%%%%%%%%%%%%%%%%%%%%%%%%%%%%%%%%
%%%%%%%%%%%%%%%%%%%%%%%%%%%%%%%%%%%%%%%%%%%%
\subsection{Background on Hausdorff dimension} \label{se:rappelhaus}
%%%%%%%%%%%%%%%%%%%%%%%%%%%%%%%%%%%%%%%%%%%%
%%%%%%%%%%%%%%%%%%%%%%%%%%%%%%%%%%%%%%%%%%%%

We recall some usual facts about Hausdorff dimensions, see \cite{Pesin}. 
\begin{definition}
Let $\mathcal F$ be a compact set of $\mathbb R^n$ and $s$ a positive real number.
\[
H^s_\delta(\mathcal F) = \inf \left\{ \sum \limits_{i=1}^{+\infty} \mbox{diam}(U_i)^s \mid \ \mbox{diam}(U_i)\leq \delta \quad and \quad \mathcal F \subset \bigcup \limits_{i=1} ^{+\infty} U_i \right\}.
\]
We introduce $H^s(\mathcal F)  = \lim \limits_{\delta \to 0} H^s_\delta(\mathcal F) \in \mathbb R^+ \cup \{+\infty\}$. 
Then there exists an unique positive real number $s_0$ such that $H^s(\mathcal F)=+\infty$ if $s> s_0$ and $H^s(\mathcal F)<+\infty$ if $s< s_0$.
This real number $s_0$ is called the {\bf Hausdorff dimension} of $\mathcal F$ and is denoted $\dim_H(\mathcal F)$.
\end{definition}

\begin{definition}
Let $\mathcal F$ be a compact set of $\mathbb R^n$. For $r>0$, we denote $\mbox{Numb}(r)$ the minimal number of balls of radius $r$ needed to cover $\mathcal F$. Then the {\bf lower and upper box-counting dimensions} are respectively defined by
\[
\underline{\dim}_B(\mathcal F)  = \liminf \limits_{r \to 0} - \frac{\ln \mbox{Numb}(r) }{\ln r}
	\quad \mbox{ and } \quad 
\overline{\dim}_B(\mathcal F) = \limsup \limits_{r \to 0} - \frac{\ln \mbox{Numb}(r) }{\ln r}.
\]
If these numbers are equal we speak about {\bf box-counting dimension}.
\end{definition}

We recall a classical result also called Frostman Lemma.
\begin{lemme}\cite{Pesin}[thm 4.4] \label{lemme:techniquedimension}
Assume there exists a mesure $\tilde \mu$ such that $\tilde \mu(\mathcal F_\infty)>0$ and such that for almost every $x$ in $\mathcal F_\infty$,
\[
\liminf \frac{ \ln \Big( \tilde \mu \big( \mathcal B(x,r) \big) \Big) }{ \ln (r)} \geq a.\]

Then we have 	$ \dim_H(\mathcal F_\infty)\geq a.$
\end{lemme}

\subsection{Technical lemmas} \label{sec:calc-hausdorff-res} 
%%%%%%%%%%%%%%%%%%%%%%%
%%%%%%%%%%%%%%%%%%%%%%%

Let $\omega=(\theta,\epsilon)\in\Omega$ with $\theta$ an irrational number.
By an abuse of notation, we will use the symbol $\omega_\ell$ to denote  $\S^\ell\omega$ and $\boldsymbol{R}^{(\ell)} = \frac{1}{\r(\omega)} \cdots \frac{1}{\r(\omega_\ell)}$.

\begin{lemme}\label{lem:afairebis}
With the notations \eqref{eq:NaNb}, for $\boldsymbol \nu$-almost all $\omega$, there exist $\boldsymbol \lambda(\omega)$ and $\boldsymbol{R}(\omega)$ such that:
\begin{eqnarray} 
&&	\lim \limits_{\ell \to +\infty} \frac{1 }{\ell}\ln  \boldsymbol N^{(\ell)}
	  = \boldsymbol\lambda(\omega). \label{eq:convergencedeN}
\\
&&\lim \limits_{\ell \to +\infty} \frac{1 }{\ell}\ln  \boldsymbol N_a^{(\ell)} =\lim \limits_{\ell \to +\infty} \frac{1 }{\ell}\ln  \boldsymbol N_b^{(\ell)} 
	=\lim \limits_{\ell \to +\infty} \frac{1 }{\ell}\ln  \boldsymbol N^{(\ell)}
	  = \boldsymbol\lambda(\omega). \label{eq:convergencedesN}
\\
&&\lim \limits_{\ell \to +\infty} \frac{1 }{\ell}\ln  \boldsymbol{R}^{(\ell)}(\omega) = \ln \boldsymbol{R}(\omega). \label{eq:convergenceder}
\end{eqnarray}

Consider the number $\boldsymbol s(\omega)$ defined $\boldsymbol \nu$-almost everywhere by 
\[
\boldsymbol s(\omega) = \frac{\boldsymbol \lambda(\omega)}{ \ln \boldsymbol{R}(\omega)}.
\] 

Then the functions $\boldsymbol \lambda (\omega)$, $\ln \boldsymbol{R}(\omega)$ and $\boldsymbol{s}(\omega)$ 
are almost everywhere constant for the measure $\boldsymbol\nu$. 
\end{lemme}
In this proof we will use Definition \ref{def:cocycles} and Theorem \ref{thm:oseledets}. Consider the cocycle defined by: 

\[
\begin{array}{lllll}
\Omega\times \mathbb{N}&\rightarrow& GL _2(\mathbb{Z}) \cap \mathcal{M}_2(\mathbb N)\\
(\omega,n)&\mapsto&  \,^t \M( \S^n \omega) \times \cdots \times  \,^t \M(  \omega) =  \,^t \M^{(n)}(\omega)
\end{array}
\]

\begin{remarque}
By Lemma \ref{lemme:cocyclemarchepas}, the map $\omega \to \ln \|M(\omega)\|$ does not belong to $L^1(\nu)$.
Thus we can not apply Oseledets Theorem, see Theorem \ref{thm:oseledets}. This explains the necessity to define the acceleration of the map $S$.
\end{remarque}

\begin{proof}
By Lemma \ref{lem:integration} the functions 
	$\omega \to \ln \| \M(\omega)\|$, $\omega \to \ln \| \M^{-1}(\omega)\|$ and $\ln \circ \  \r$
belong to  $L^1(\boldsymbol \nu)$. Thus the same is true for $\omega \to \ln \| \,^t\M(\omega)\|$ and $\omega \to \ln \| \,^t \M^{-1}(\omega)\|$.

By Theorem \ref{thm:oseledets} of Oseledets, we obtain the convergence for almost every $\omega$ of:
\[
 \frac{1}{\ell} \ln  \left\| \,^t \M^{(\ell)}(\omega) \binom11 \right\|_1.
\]

Now an easy computation shows that for each integer $\ell$:
%since for each $\omega$, $\M(\omega) \in\mathcal M_2(\mathbb N)$, 
\[
\left\| \,^t \M^{(\ell)}(\omega) \binom11 \right\|_1 =   \left\| (1,1) \times \,^t \M^{(\ell)}(\omega)  \right\|_1 
	=   \left\| \,^t \left( \M^{(\ell)}(\omega) \binom11\right)  \right\|_1  
	=   \left\|   \M^{(\ell)}(\omega) \binom11  \right\|_1 .
\]
This expression is equal to $\m_{1,1}^{(\ell)}(\omega)+\m_{1,2}^{(\ell)}(\omega)+\m_{2,1}^{(\ell)}(\omega)+\m_{2,2}^{(\ell)}(\omega)$.

By Proposition \ref{prop:mes-inv-accS} the dynamical system is ergodic. From \eqref{equationconepos} and from Furstenberg-Kesten Theorem and its corollary \ref{co:FK}, the sequences of Equality \eqref{eq:convergencedeN} converge for almost every $\omega$ and 
\[
\lim \limits_{\ell \to +\infty} \frac{1 }{\ell}\ln  \boldsymbol N^{(\ell)}
	= \lim \limits_{\ell \to +\infty} \frac{1 }{\ell}\ln \left\|   \M^{(n)}(\omega) \binom11  \right\| 
	= \lim \limits_\ell   \frac{1 }{\ell} \int \ln \|\,^t \M^{(\ell)} \| d \boldsymbol \nu
	= \lim \limits_\ell   \frac{1 }{\ell} \int \ln \|  \M^{(\ell)} \| d \boldsymbol \nu.
\]

The convergence of Expression \eqref{eq:convergenceder} is a direct application of the Birkhoff Theorem. 

It remains to prove the convergence of the terms in Equation \eqref{eq:convergencedesN}. Let us equipe $\mathbb R^2$ with the partial order defined by:  $z=(x,y)\leq z'=(x',y')$
if $x\leq x'$ and $y\leq y'$.
By Definition \ref{defcocycleacc} the matrices $\M(\omega)$ have positive coefficients. Thus if $0\leq z\leq z'$, we obtain for all $\omega$:  
$\|\M (\omega) z\|_1 \leq \| \M(\omega) (z')\|_1$.
We deduce immediately that for each $\ell$
\begin{equation} \label{eq:premieremaj}
\left\|   \M^{(\ell)}(\omega) \binom10  \right\|_1  \leq  \left\|   \M^{(\ell)}(\omega) \binom11  \right\|_1
	 \quad \mbox{ and } \quad
\left\|   \M^{(\ell)}(\omega) \binom01  \right\|_1  \leq  \left\|   \M^{(\ell)}(\omega) \binom11  \right\|_1.
\end{equation}

Now we use the correspondence between $x\in[0,2]$ and $\omega\in \Omega$. 
Let us consider three cases:

\begin{itemize}
\item If $x\in(0,1]$, then with $n=\floor{\frac1x}\geq 1$ :
\[
   \M (x) \binom10   = \binom{2n-1}{n} \geq \binom11
\quad \mbox{ and }\quad
   \M (x) \binom01   = \binom{2}{1} \geq \binom11.
\]
\item If $x\in(1,3/2)$, then with $n=\floor{\frac1{x-1}}\geq 2$ :
\[
   \M (x) \binom01   = \binom{2(n-1)}{1} \geq \binom11
\quad \mbox{ but we also have } \quad
   \M (x) \binom10   = \binom{1}{0}. 
\]

\item If $x\in(3/2,2]$, then with $n=\floor{\frac1{2-x}}\geq 2$ :
\[
   \M (x) \binom10   = \binom{2n-1}{n-1} \geq \binom11
\quad \mbox{ and } \quad
   \M (x) \binom01   = \binom{2}{1} \geq \binom11.
\]

\end{itemize}
By definition of $\S$ we remark that if $x\in [1,3/2]$, then $\S(x) \notin (1,3/2)$.
We deduce from the two preceding inequalities:
\[
\M(\S x)   \M (x) \binom10   \geq \M(\S x) \binom{1}{0} \geq \binom11. 
\]
%Finally we can deduce in all the cases:

With Equation \eqref{eq:premieremaj}, we deduce for all integer $\ell \geq 2$ :
\[
\begin{array}{ll}
 \left\|   \M^{(\ell-2)}(\omega) \binom11  \right\|_1 \leq \left\|   \M^{(\ell)}(\omega) \binom10  \right\|_1  \leq  \left\|   \M^{(\ell)}(\omega) \binom11  \right\|_1 ,\\
 \left\|   \M^{(\ell-2)}(\omega) \binom11  \right\|_1 \leq  \left\|   \M^{(\ell)}(\omega) \binom01  \right\|_1  \leq  \left\|   \M^{(\ell)}(\omega) \binom11  \right\|_1. \\
\end{array}
\]

Finally we have:
\begin{eqnarray} \label{eq:encadrementdesN}
\frac{1}{\ell} \ln  \left\|   \M^{(\ell-2)}(\omega) \binom11  \right\|_1 \leq \frac{1}{\ell} \ln   \left\|   \M^{(\ell)}(\omega) \binom10  \right\|_1 
	 \leq \frac{1}{\ell}  \ln   \left\|   \M^{(\ell)}(\omega) \binom11  \right\|_1 \\
\frac{1}{\ell}\ln     \left\|   \M^{(\ell-2)}(\omega) \binom11  \right\|_1 \leq  \frac{1}{\ell}  \ln \left\|   \M^{(\ell)}(\omega) \binom01  \right\|_1  
	\leq  \frac{1}{\ell} \ln   \left\|   \M^{(\ell)}(\omega) \binom11  \right\|_1.
\end{eqnarray}
Since the left and right terms converge to $\boldsymbol \lambda(\omega)$, we deduce the convergence of the terms of Equation \eqref{eq:convergencedesN}.

The ergodicity of the map, by Proposition \ref{prop:mes-inv-accS} allows to conclude that the functions  $\boldsymbol \lambda (\omega)$, $\ln \boldsymbol{R}(\omega), \boldsymbol{s}(\omega)$ are almost everywhere constant.
\end{proof}

\begin{definition}\label{def:constant}
We denote  respectively by 
$ \boldsymbol{\lambda}$, $\ln \boldsymbol{R}$ the constant functions associated to  $\boldsymbol \lambda (\omega)$, $\ln \boldsymbol{R}(\omega)$.
We also denote $\boldsymbol{s}= \frac{\boldsymbol \lambda}{ \ln \boldsymbol{R}}$.
We finally define $\Omega'$ as the subset of $\Omega$ for which the previous expressions converge.
\end{definition}

Remark that it is straightforward to check that the expressions converge if $\omega$ has an ultimatelly periodic $S$ expansion.

The main objective of the next two subsections is the computation of the Hausdorff dimension of $\mathcal K_\omega$. We shall prove it is equal to $\boldsymbol s$. 

\subsection{Minoration of the Hausdorff dimension in Theorem \ref{thm:lyap}}  \label{subse:minoration}
%%%%%%%%%%%%%%%%%%%%%%%
%%%%%%%%%%%%%%%%%%%%%%

We use the measure $\mu$ on $\mathcal K_\omega$ of Definition \ref{def:mesure}.

\begin{lemme} \label{lemme2}
For $\omega\in\Omega'$, we have for each $x\in \mathcal K_\omega$ :
\begin{equation} \label{eq:liminf} 
\liminf \limits_{r \to 0} \frac{\ln \mu \big( \mathcal B(x,r) \cap \mathcal K_\omega\big)}{\ln r} \geq \boldsymbol{s}.
 \end{equation}
\end{lemme}
\begin{proof}
Let $x$ be an element of $\mathcal K_\omega$ and $r>0$. 
The map $h: \Sigma_\omega \mapsto \mathcal K_\omega$ is onto by Proposition \ref{def:h}, thus there exists an element $u\in\Sigma_\omega$ such that $x=h(u)$. Now consider the integer $\ell$ such that ${\boldsymbol R}^{(\ell)}\leq \sqrt 2 \cdot r \leq {\boldsymbol R}^{(\ell-1)}$.
By Lemma \ref{lem:partitionl} there exists an element $C\in \mathcal P_\ell$ such that $u\in C$. Moreover the set $h(C)$ is included either in a rectangle of sides $({\boldsymbol R}^{(\ell)},\theta_\ell {\boldsymbol R}^{(\ell)})$, or in a square of side ${\boldsymbol R}^{(\ell)}$. This polygon contains $x$. Thus, in any case, it is included inside $\mathcal B(x,\sqrt 2 r)$. We deduce $C \subset h^{-1} \big( \mathcal B(x,\sqrt 2 r) \cap \mathcal K_\omega\big)$ and
$$
\mu \big( \mathcal B(x,\sqrt 2 r) \cap \mathcal K_\omega\big) = \tilde\mu_\omega \Big( h^{-1} \big( \mathcal B(x,\sqrt 2 r) \cap \mathcal K_\omega\big)\Big) \geq \tilde \mu_\omega (C).
$$

By definition of $C$, the result of Lemma \ref{lem:partitionl} implies that $\tilde\mu_\omega (C)$ is equal to $\alpha_\ell$ or $\beta_\ell$.
We will consider two cases.

\paragraph{} Assume first that $ \tilde \mu_\omega (C) =\alpha_\ell$. Then we have:
\[
\dfrac{ \ln \Big( \mu \big( \mathcal B(x,\sqrt 2 r)  \cap \mathcal K_\omega\big) \Big)  }{\ln (\sqrt 2 r)}  \geq\dfrac{\ln \alpha_{\ell}  }{\ln {\boldsymbol R}^{(\ell-1)}}  
	 = \dfrac{\ln \alpha_{\ell} {\boldsymbol N}^{(\ell)}_a }{\ln {\boldsymbol R}^{(\ell-1)}}  -\dfrac{\ln {\boldsymbol N}^{(\ell)}_a }{\ln {\boldsymbol R}^{(\ell-1)}}
	  = \dfrac{\ln \big(1-\beta_{\ell} {\boldsymbol N}^{(\ell)}_b \big) }{\ln {\boldsymbol R}^{(\ell-1)}}  -\dfrac{\ln {\boldsymbol N}^{(\ell)}_a }{\ln {\boldsymbol R}^{(\ell-1)}}.
\]
Since $1-\beta_\ell {\boldsymbol N}^{(\ell)}_b<1$ and ${\boldsymbol R}^{(\ell-1)}<1$ we deduce:
\[
\dfrac{1}{\ln (\sqrt 2 r)} \ln \Big( \mu \big( \mathcal B(x,\sqrt 2 r) \cap \mathcal K_\omega\big) \Big)   	   \geq -\dfrac{\ln {\boldsymbol N}^{(\ell)}_a }{\ln {\boldsymbol R}^{(\ell-1)}}. 
\]

\paragraph{} The same proof works if $ \tilde \mu_\omega (C) =\beta_\ell$ and we obtain
\[
\dfrac{1}{\ln (\sqrt 2 r)} \ln \Big( \mu \big( \mathcal B(x,\sqrt 2 r) \cap \mathcal K_\omega\big) \Big)   	   
	\geq -\dfrac{\min(\ln {\boldsymbol N}^{(\ell)}_a ,\ln {\boldsymbol N}^{(\ell)}_b)}{\ln {\boldsymbol R}^{(\ell-1)}}. 
\]

\paragraph{}Finally we deduce 
$$\liminf \limits_{r \to 0} \frac{\ln \mu \big( \mathcal B(x,r) \cap \mathcal K_\omega\big)}{\ln r}\geq
\liminf  \limits_\ell -\dfrac{\min(\ln {\boldsymbol N}^{(\ell)}_a ,\ln {\boldsymbol N}^{(\ell)}_b)}{\ln {\boldsymbol R}^{(\ell-1)}}.
$$
Due to Lemma \ref{lem:afairebis}, if $\omega \in \Omega'$,  the expressions converge to the same value $ \boldsymbol s(\omega)$.
\end{proof}
Then by Lemma \ref{lemme:techniquedimension} we obtain: 
\begin{corollary}\label{lemme2}
For $\omega\in\Omega'$ we have $\dim_{ H}(\mathcal K_\omega)  \geq  \boldsymbol s.$
\end{corollary}

\subsection{Majoration of the Hausdorff dimension in Theorem \ref{thm:lyap}}   \label{subse:majoration}
%%%%%%%%%%%%%%%%%%%%%%%
%%%%%%%%%%%%%%%%%%%%%%

We refer to Appendix \ref{se:rappelhaus} for a quick background on the different notions and notations.
\begin{lemme}  \label{lemme1}
For $\omega\in\Omega'$ :
	 $\dim_ H(\mathcal K_\omega) \leq \underline{\dim}_B(\mathcal K_\omega)  \leq \overline{\dim}_B(\mathcal K_\omega)\leq   \boldsymbol s$.
\end{lemme}
\begin{proof}
We want to obtain a majoration of $- \liminf \limits_{r \to 0} \dfrac{\ln \mbox{\small Numb}(r) }{\ln r}$.

Remark that if $0<r<s$ and $\mathcal K_\omega \subset \cup \mathcal B(x_i,r)$, then  $\mathcal K_\omega \subset \cup \mathcal B(x_i,s)$ and thus we have Numb$(r)\geq$ Numb$(s)$.
Let us fix $r>0$ and denote $\ell$ the integer such that $\boldsymbol{R}^{(\ell)} \leq r\leq\boldsymbol{R}^{(\ell-1)}$.
Now we claim the following fact (proved at the end)
$$\mbox{Numb }( \boldsymbol{R}^{(\ell)} )\leq  \ln \left\| \M(\S^\ell \omega) \times \cdots \times \M(\omega) \binom11\right\|_1
	=  \ln \N^{(\ell)}(\omega).$$

We deduce for $\boldsymbol{R}^{(\ell)} \leq r \leq \boldsymbol{R}^{(\ell-1)}$ :
\[
 -\frac{\ln \mbox{Numb}(r) }{\ln r} \leq  -\frac{\ln  \N^{(\ell)} }{\ln \boldsymbol{R}^{(\ell-1)}}.
 \]

obtaining
\[
\overline{\dim}_B(\mathcal K_\omega) \leq -\limsup \frac{\ln  \N^{(\ell)}}{\ln \boldsymbol{R}^{(\ell-1)}}.
\]

\textit{We finish with a proof of the claim.}
By Lemma \ref{lemme:recouvrement} the set $\mathcal K_\omega (\boldsymbol p_\ell)$ can be covered by $\m_{1,1}^{(\ell)}(\omega)+\m_{2,1}^{(\ell)}(\omega)$ images of $\mathcal C$ and $\m_{1,2}^{(\ell)}(\omega)+\m_{2,2}^{(\ell)}(\omega)$ images of $\mathcal R_{\theta_\ell}$ by some similitudes of ratio $\boldsymbol{R}^{(\ell)}$. Since $\mathcal R_{\theta_\ell}$ can be covered by one square of size $1$. We deduce that $\mathcal K_\omega (\boldsymbol p_\ell)$ can be covered by $\N^{(\ell)}(\omega)$ squares of size $\boldsymbol{R}^{\ell}$. Each such square is covered by one ball of the same radius, and we have $\mathcal K_\omega \subset \mathcal K_\omega (\boldsymbol p_\ell)$. 
This finishes the proof of the claim.
\end{proof}

\subsection{Conclusion}

\begin{proposition}
For all $\omega \in \Omega'$, $\dim_ H(\mathcal K_{\omega})= \boldsymbol s$.
\end{proposition}

\begin{proof}
We have proved the lower bound in Corollary \ref{lemme2} and the upper bound in Lemma \ref{lemme1}.
\end{proof}
%%%%%%%%%%%%%%%%%%%%%%%%%%%%%%%%%%%%%%%%%%%%
%%%%%%%%%%%%%%%%%%%%%%%%%%%%%%%%%%%%%%%%%%%%
\section{Numerical values}\label{sec:numlyap}
%%%%%%%%%%%%%%%%%%%%%%%%%%%%%%%%%%%%%%
%%%%%%%%%%%%%%%%%%%%%%%%%%%%%%%%%%

\subsection{For $\boldsymbol{\nu}$-almost paramter}
%%%%%%%%%%%%%%%%%%%%%%%%%%%%%%%
%%%%%%%%%%%%%%%%%%%%%%%%%%%%%%%

%On va prouver le deuxieme point du gros theoreme. Il repose sur les deux lemmes suivants :

\begin{lemme}\label{lem:majoration-Lyap}
For $\boldsymbol \nu$-almost $\omega\in\Omega$, 
$$
\boldsymbol \lambda  \leq \int \ln  \| \M(\omega) \|_\infty d \boldsymbol \nu(\omega) \leq 3.8.
$$
\end{lemme}

\begin{proof}
Let $u\in (\mathbb R_+)^2$ and $\omega\in\Omega'$ such that 
\[
\frac{1}{\ell} \sum \limits_{k=0}^\ell \ln \| \M(\S^k \omega)\|_\infty \to \int \ln  \| \M(\omega) \|_\infty d \boldsymbol \nu(\omega)
\mbox{ and }
\frac{1}{\ell}   \ln \|  \M^{(\ell)}(\omega)u  \|_\infty \to  \boldsymbol \lambda.
\]
Then, for each $\ell$
\[
\frac{1}{\ell}   \ln \|  \M^{(\ell)}(\omega) u \|_\infty  \leq 
	\frac{1}{\ell}   \ln \| u \|_\infty + \frac{1}{\ell} \sum \limits_{k=0}^\ell \ln \| \M(\S^k \omega)\|_\infty.
\]
The limit of the different terms when $\ell$ tends to infinity gives the result. We refer to Lemma \ref{lem:integration} for the computation of the integrals. 
\end{proof}

First we recall a result of \cite{thieullen} and \cite{Wojt} (Proposition 1.11 p 25):
\begin{lemme}
Consider the function
$$f:\begin{array}{ccc}
\Omega&\rightarrow&\mathbb R\\
\omega&\mapsto&  \sqrt{\m_{1,1}(\omega)\m_{2,2}(\omega)}+ \sqrt{\m_{1,2}(\omega)\m_{2,1}(\omega)}.
\end{array}
$$
Then we have
\[
\boldsymbol \lambda = \lim \limits_\ell \frac1\ell \int \ln \| \M^{(\ell)}(\omega) \| d \boldsymbol \nu (\omega) \geq \int \ln f(\omega) d \boldsymbol \nu (\omega).
\]
\end{lemme}

\begin{corollaire}\label{lem:min-Lyap}
For $\boldsymbol \nu$-almost $\omega\in\Omega$, we obtain
$$
	2.66 \leq \int \ln  f(\omega) d  \boldsymbol  \nu(\omega)   \leq  \boldsymbol \lambda.
$$
\end{corollaire}
\begin{proof}
Due to Proposition \ref{prop:mes-inv-accS} the system is ergodic, then we apply Furstenberg-Kesten theorem (Theorem \ref{thm:FK} and Corollary \ref{co:FK}) in order to have:
\[
\boldsymbol \lambda(\omega) = \lim \limits_\ell \frac1\ell \int \ln \| \M^{(\ell)}(\omega) \| d \boldsymbol \nu (\omega).
\]

Now we use preceding Lemma, the rest of the proof is a simple computation.

\end{proof}

\begin{corollary}
We obtain for $\boldsymbol \nu$ almost all $\omega$:
$$1.07\leq \boldsymbol s\leq 1.55.$$
\end{corollary}
\begin{proof}
Recall that by Definition \ref{def:constant} we have $\boldsymbol s= \frac{\boldsymbol \lambda}{ \ln \boldsymbol{R}}$.
We use  Lemma \ref{lem:majoration-Lyap}, the numerical values are obtained from Equation \eqref{equationintegraleretM} in Lemma \ref{lem:integration}. We deduce:
$$ 
	\boldsymbol s  \leq \frac{\int \ln  \| \M(\omega) \|_\infty d \boldsymbol \nu(\omega)}{\int \ln \big( \boldsymbol{R}(\omega) \big) d \boldsymbol \nu(\omega)} 
	\leq \frac{3.8}{2.46} \leq 1.55.
$$

Then Lemma \ref{lem:min-Lyap} gives:

$$
		1.07 \leq  \frac{2.66}{2.47} \leq \frac{\int \ln  f(\omega)  d \boldsymbol \nu(\omega)}{\int \ln \big( \boldsymbol{R}(\omega) \big) d \boldsymbol \nu(\omega)}  
	\leq \boldsymbol s.
$$

\end{proof}

\subsection{Self similar points}
%%%%%%%%%%%%%%%%%%%%%
%%%%%%%%%%%%%%%%%%%%%%%
A straightforward computation shows that the fixed points of $S$ and $\S$ are of the following forms:
\[
\left( \sqrt{n^2+1}-n ,-1\right) \quad \text{and} \quad  \left(\sqrt{n(n+2)}-n ,1\right) \quad \text{for} \ \ n\in\mathbb N.
\]

Let us define the sequences $\alpha_n  = \frac{\sqrt{n^2+4}-n}{2}, \beta_n =\frac{  1-n+\sqrt{(n-1)(n+3)} }{2}$. 
This allows us to denote these fixed points as
\[ (\alpha_{2n},-1) \quad \text{and} \quad (\beta_{2n+1},1). \]

%On a $n_{(\alpha_n,-1)} = n_{(\beta_{n},1)}=n$.

Let us compute the Hausdorff dimension for the first family. We apply Theorem \ref{thm:lyap}.
\[\dim_{ H}(\mathcal K_{(\alpha_{2n},-1)})=
\lim_{\ell} -\dfrac{ \ln \left\| M ^{(\ell)} (\alpha_{2n},-1) \right \|}{\ell \ln(\alpha_{2n})}
=
\lim_{\ell} -\dfrac{ \ln \left\| M ^\ell (\alpha_{2n},-1) \right \|}{\ell \ln(\alpha_{2n})}
\]

If $\lambda_\omega$ is the dominant eigenvalue of the matrix $M(\omega)$ we obtain
$$\dim_{ H}(\mathcal K_\omega)=-\dfrac{ \ln |\lambda_\omega |}{ \ln(\alpha_{2n})}.$$

We conclude by the computation of the dominant eigenvalue of 
$\begin{pmatrix} 2n -1 &2 \\ n  &1 \end{pmatrix}$.

The next figure shows the numerical values of the dimension. We can compare with the second statement of Theorem \ref{thm:lyap}.
\begin{figure}[H]
\[
\begin{array}{| c | c | c | c | c | c |} 
\hline
n & 1 & 2 & 3 & 4 & 5 \\
\hline
\mbox{ Hausdorff dimension of $\mathcal K_{(\alpha_{2n},-1)}$} & 
	  1.637\ 938    & 1.450 \ 998   & 1.370\ 279   & 1.325\ 467   & 1.296\ 563  \\
\hline
\mbox{ Hausdorff Dimension of $\mathcal K_{(\beta_{2n+1},1)}$} & 
	  1.338\ 499 &    1.300 \ 488  &   1.276\ 470 &    1.259\ 479  &    1.246 \ 613 \\
\hline
\end{array}
\]
\caption{Approximation of the Hausdorff dimensions.}
\end{figure}

\begin{figure}[H]
\includegraphics[width=14cm]{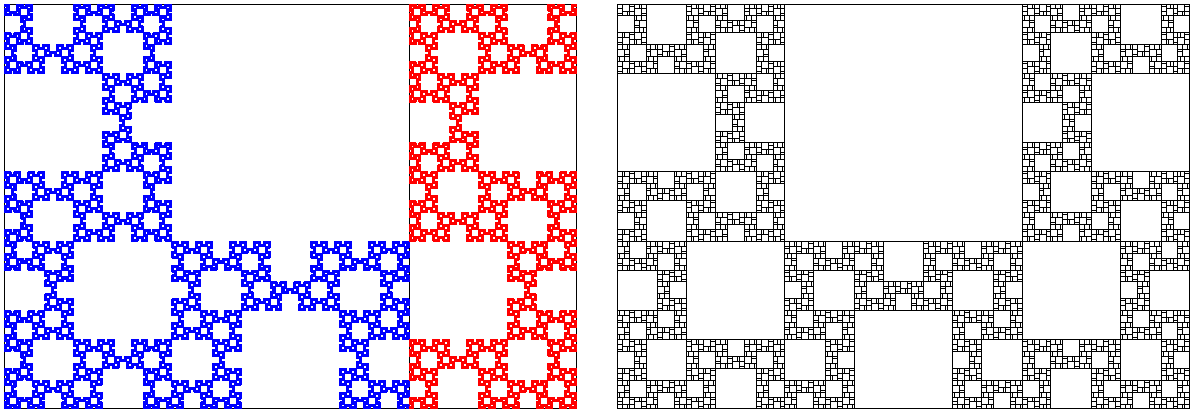}
\caption{The sets $\mathcal K_\omega$ and $\mathcal D_\omega$ for $\omega=(\alpha_2,-1)$.}
\end{figure}

\begin{figure}[H]
\includegraphics[width=14cm]{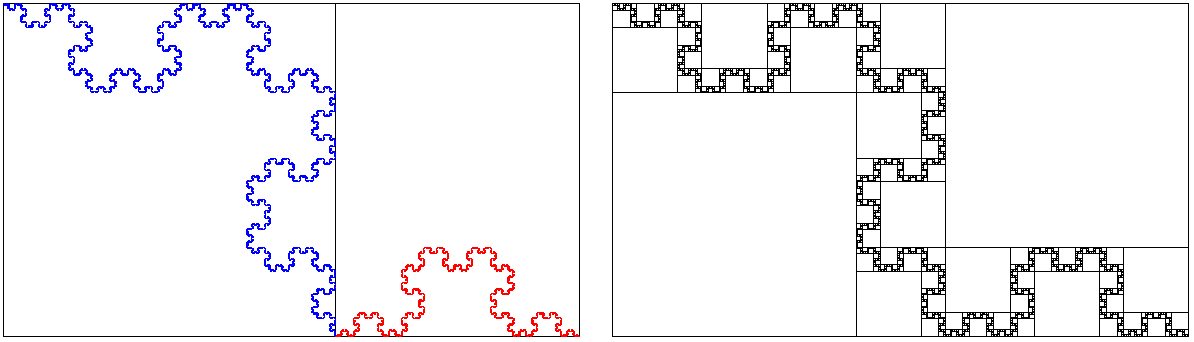}
\caption{The sets $\mathcal K_\omega$ and $\mathcal D_\omega$ for $\omega=(\beta_3,1)$.}
\end{figure}

%%%%%%%%%%%%%%%%%%%%%%%%%%%%%%%%%%%%%%%%%%%%
%%%%%%%%%%%%%%%%%%%%%%%%%%%%%%%%%%%%%%%%%%%%
\bibliographystyle{plain}
\bibliography{biblio-pet}
%%%%%%%%%%%%%%%%%%%%%%%%%%%%%%%%%%%%%%%%%%%%
%%%%%%%%%%%%%%%%%%%%%%%%%%%%%%%%%%%%%%%%%%%%

\appendix

%%%%%%%%%%%%%%%%%%%%%%%%%%%%%%%%%%%%%%%%%%%%
%%%%%%%%%%%%%%%%%%%%%%%%%%%%%%%%%%%%%%%%%%%%
\section{The renormalization map}
%%%%%%%%%%%%%%%%%%%%%%%%%%%%%%%%%%%%%%%%%%%%
%%%%%%%%%%%%%%%%%%%%%%%%%%%%%%%%%%%%%%%%%%%%

\subsection{Invariant measure for a continued fraction algorithm} \label{subs:desciptionalgo}
%%%%%%%%%%%%%%%%%%%%%%%
%%%%%%%%%%%%%%%%%%%%%

We recall the method introduced by Arnoux and Nogueira, \cite{Arn.Nog.93} and developed  in Arnoux-Schmidt, see \cite{Arn.Sch.13}.
We consider a measure-preserving dynamical system $(T, I, B, \nu)$, where $T: I \rightarrow I$ is a measurable map on the measurable space $(I,B)$ which preserves the measure
$\nu$ ($\nu$ is usually a probability measure).
A natural extension of the dynamical system $(T,I,B,\nu)$ is an invertible system $(\tilde T,\Theta,B',\mu)$ with a surjective projection $\pi : \Theta \rightarrow I$ making $(T, I, B, \nu)$ a factor of $(\tilde T,\Theta,B',\mu)$, and such that any other invertible system with this property has its projection factoring through $(\tilde T,\Theta,B',\mu)$. The natural extension of a dynamical system exists always, and is unique up to measurable isomorphism, see \cite{Roh}. Informally, the natural extension is given by appropriately giving to (forward) $T$-orbits an infinite (in general) past; an abstract model of the natural extension is easily built using inverse limits.
There is an efficient heuristic method for explicitly determining a geometric model of the natural extension of an interval map when this map is given (piecewise) by Mobius transformations. If the map is given by generators of a Fuchsian group of finite covolume, then one can hope to realize the natural extension as a factor of a section of the geodesic flow on the unit tangent bundle of the hyperbolic surface uniformized by the group. 

\begin{definition}
An interval map $f:I\rightarrow I$ is called a {\bf piecewise Mobius map} if there is a partition of $I$ into intervals $I=\bigcup I_\alpha$  and a set $M:=\{M(\alpha)\}$ of elements $PSL(2, \mathbb R)$ such that the restriction of $f$ to $I_\alpha$ is exactly given by $x\in I_\alpha\rightarrow M(\alpha).x$ . We call the subgroup of $PSL(2, \mathbb R)$ generated by $M$, the group generated by $f$ and denote if by $\Gamma_f$ .
\end{definition}

We will use the classical result, see \cite{Arn.Nog.93}:

\begin{proposition}\label{Sergodique}
Assume $d\nu$ is a finite non zero measure, $S$ invariant, such that the measure of a set is equal to the measure of its closure.
There exists $\Theta\subset \mathbb{R}^2$ such that dynamical system $(\tilde{S},\Theta,d\mu) $ is a natural extansion of the dynamical system $(S,[0,2],d\nu)$ where:
\begin{itemize}
\item The map $\tilde{S}$ is piecewise defined by the following formula where $A\in SL_2(\mathbb Z)$
$$\begin{array}{cccc}
\tilde{S}:&\Theta&\rightarrow& \Theta\\
&(x,y)&\mapsto &\left(A(x).x,\dfrac{-1}{A(x).(-1/y)} \right)\\
\end{array}$$

\item The map $\tilde{S}$ has an invariant mesure given by $d\mu=\dfrac{dxdy}{(1+xy)^2}$.
 \end{itemize}
\end{proposition}

\subsection{Invariant measure for the accelerated renormalization map} \label{subse:mesinv}
%%%%%%%%%%%%%%%%%%%%%%%%
%%%%%%%%%%%%%%%%%%%%

It remains to find a domain $\boldsymbol \Theta$ where this map is bijective.
The following figure describes this domain:

\begin{figure}
 \includegraphics[width=8cm]{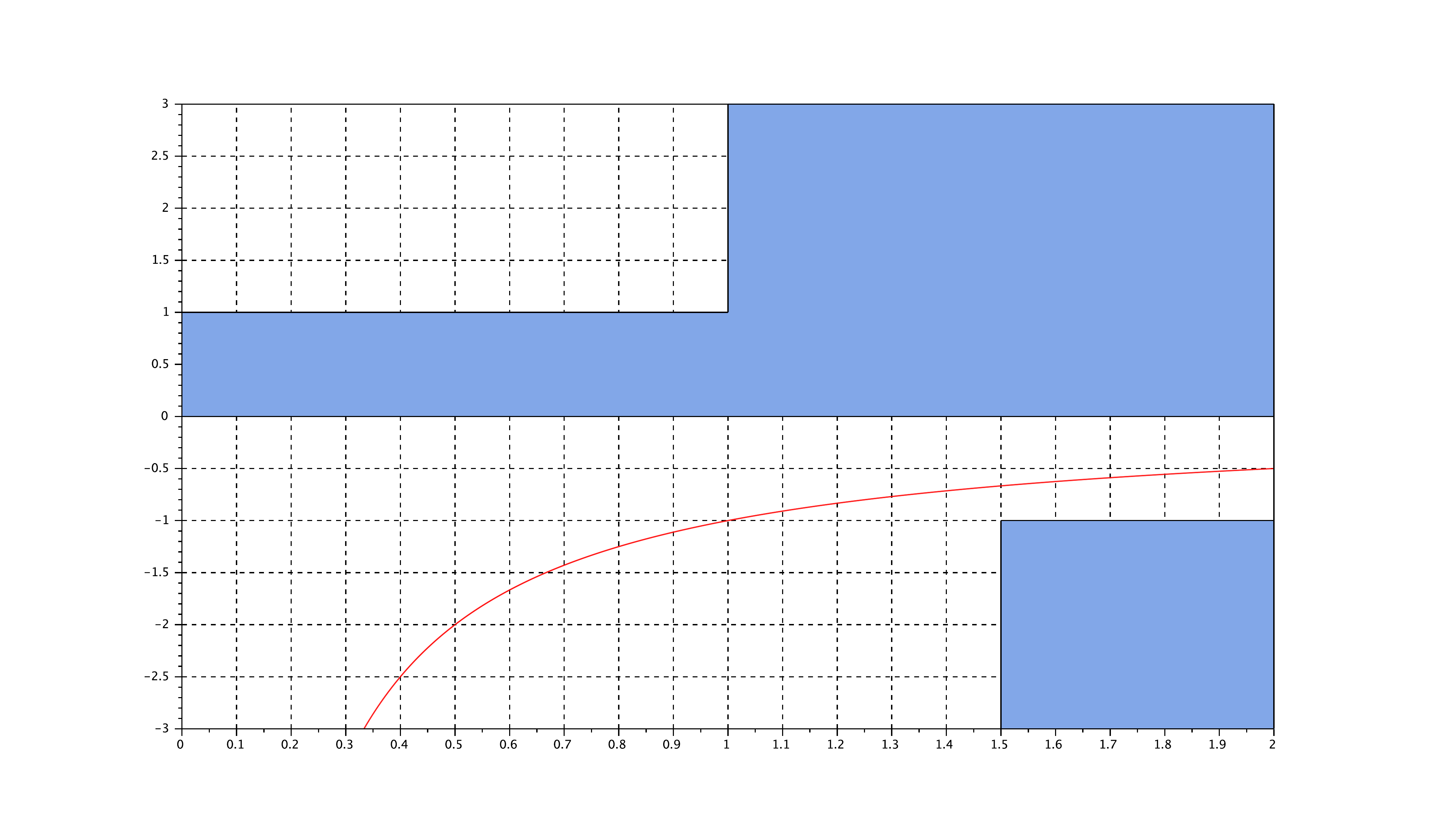}
\caption{A part of the domain $\boldsymbol \Theta$ and the red curve  $y=-1/x$.}
\end{figure}

\begin{center}
\begin{tikzpicture}
  \node (A) {$\Theta$};
  \node (D) [right of=A, node distance=4cm] {$\Theta$};
  \node (B) [below of=A, node distance=2cm] {$[0,2]$};
  \node (C) [right of=B, node distance=4cm] {$[0,2]$};
  \draw[->,dashed] (A) to node [left] {$\pi$} (B);
  \draw[->] (B) to node [below] {$\S$} (C);
  \draw[->] (A) to node [above] {$\tilde{\S}$} (D);
  \draw [->, dashed] (D) to node[right] {$\pi$} (C);
\end{tikzpicture}
\end{center}

\begin{lemme}
The domain $\boldsymbol \Theta=[0,1]\times [0,1]\cup [1,2] \times \mathbb R^+\cup[3/2,2]\times [-1,+\infty)$ is invariant for the application  $\tilde \S$ defined by $(x,y)\mapsto \left(\A(x).x, \dfrac{-1}{{\A}(x).(-\frac{1}{y})}\right).$ Moreover the map $\tilde \S : \boldsymbol \Theta \mapsto \boldsymbol \Theta$ is bijective.
\end{lemme}
\begin{proof}
The proof is based on the following diagrams. 
The three following images form a partition of $\Theta$. 
\begin{figure}[H]
\includegraphics[width=5cm]{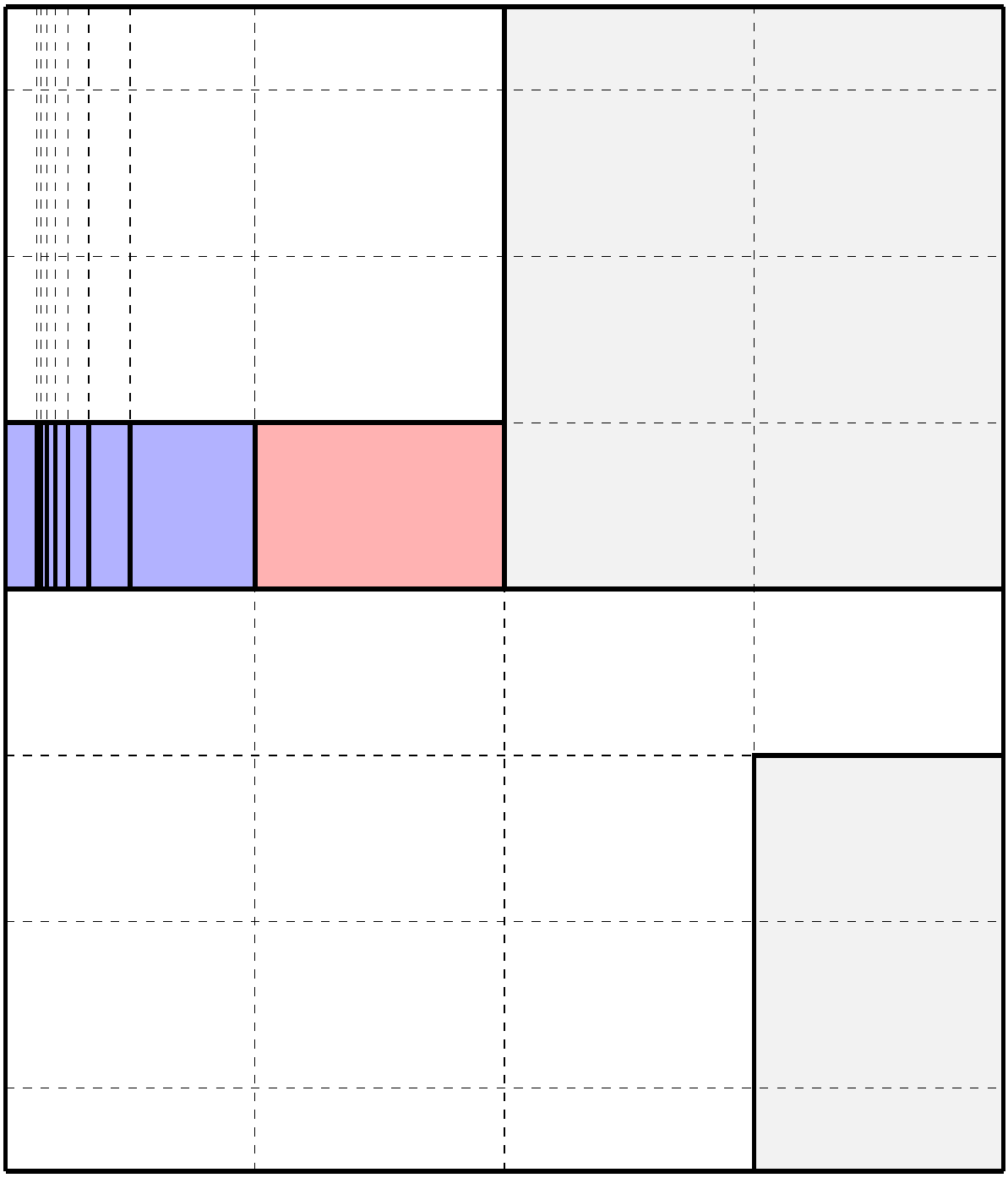}
\includegraphics[width=5cm]{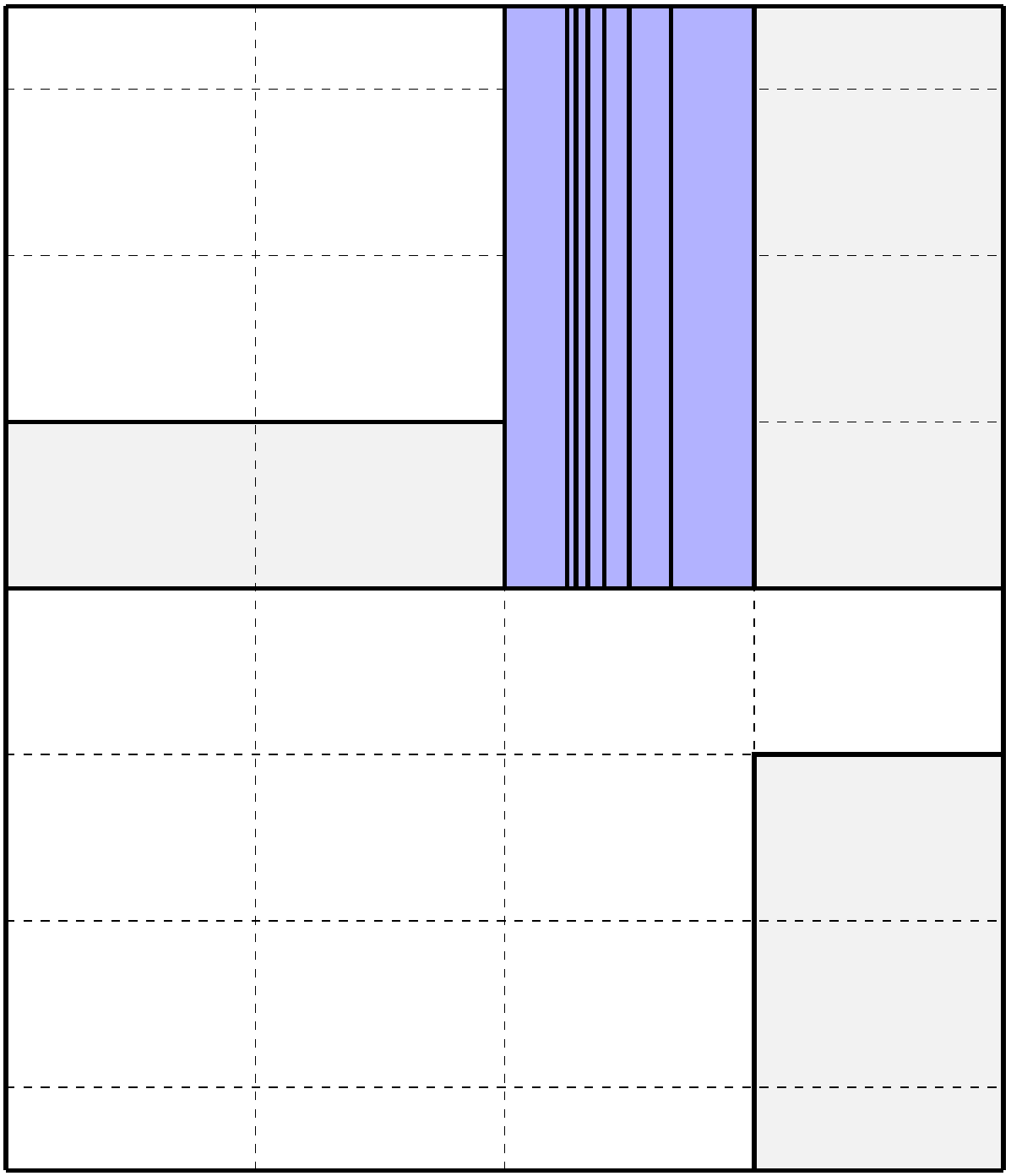}
\includegraphics[width=5cm]{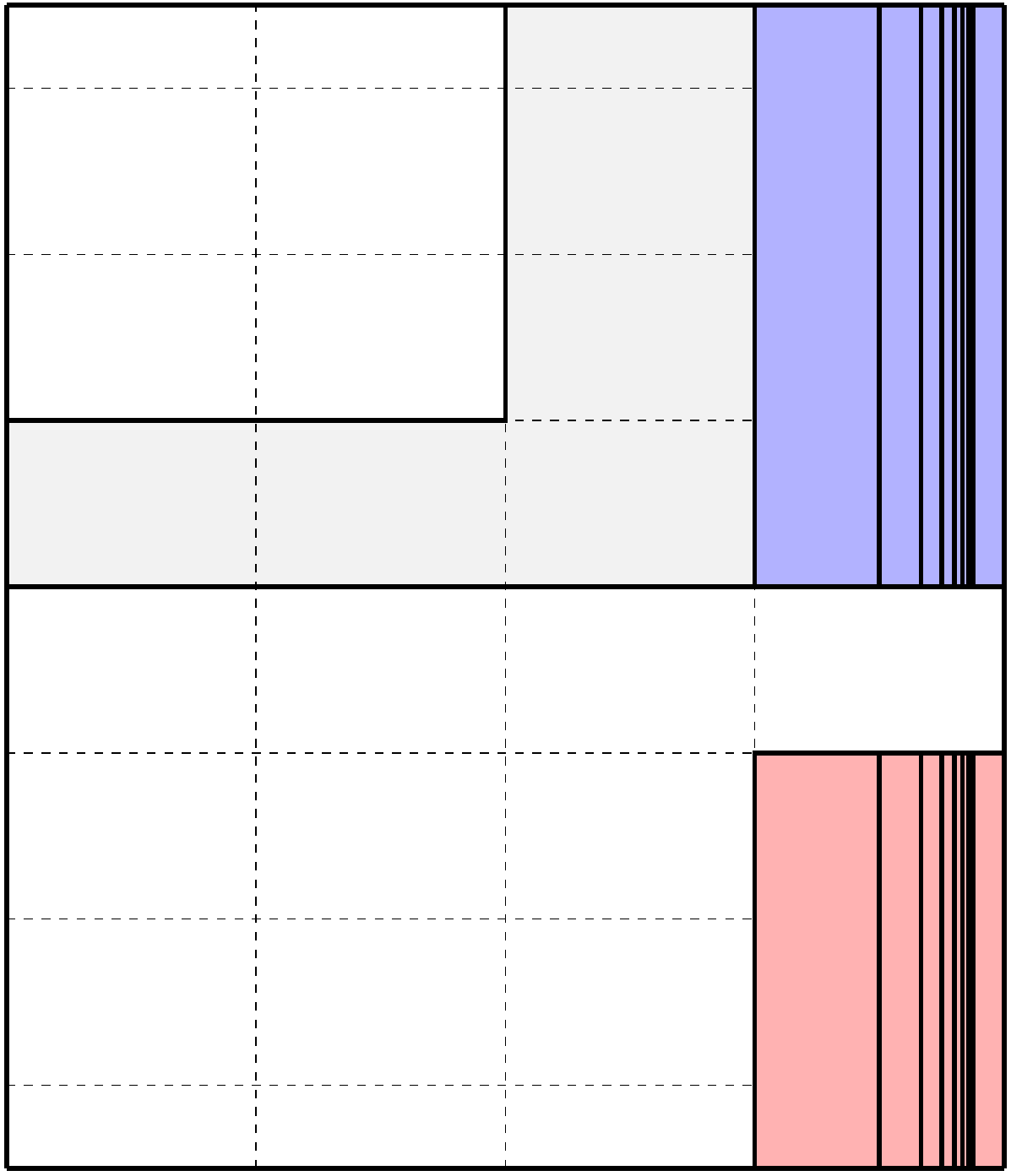}
\end{figure}

We represent in the following images, the corresponding images by $\tilde{\S}$.
They also form a partition of $\Theta$. 
This proves that $\tilde \S$ is a bijection on $\Theta$.

\begin{figure}[H]
\includegraphics[width=5cm]{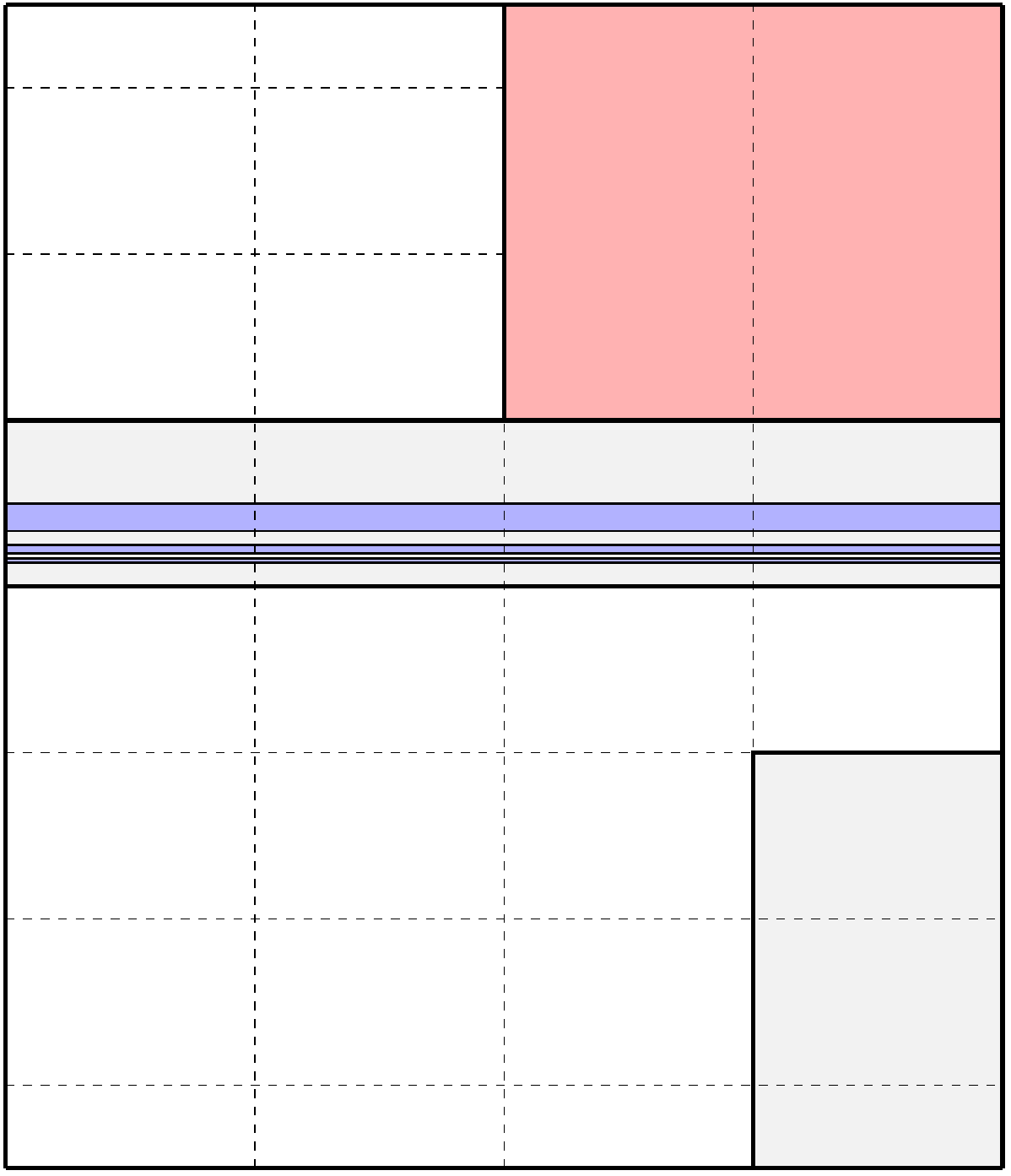}
\includegraphics[width=5cm]{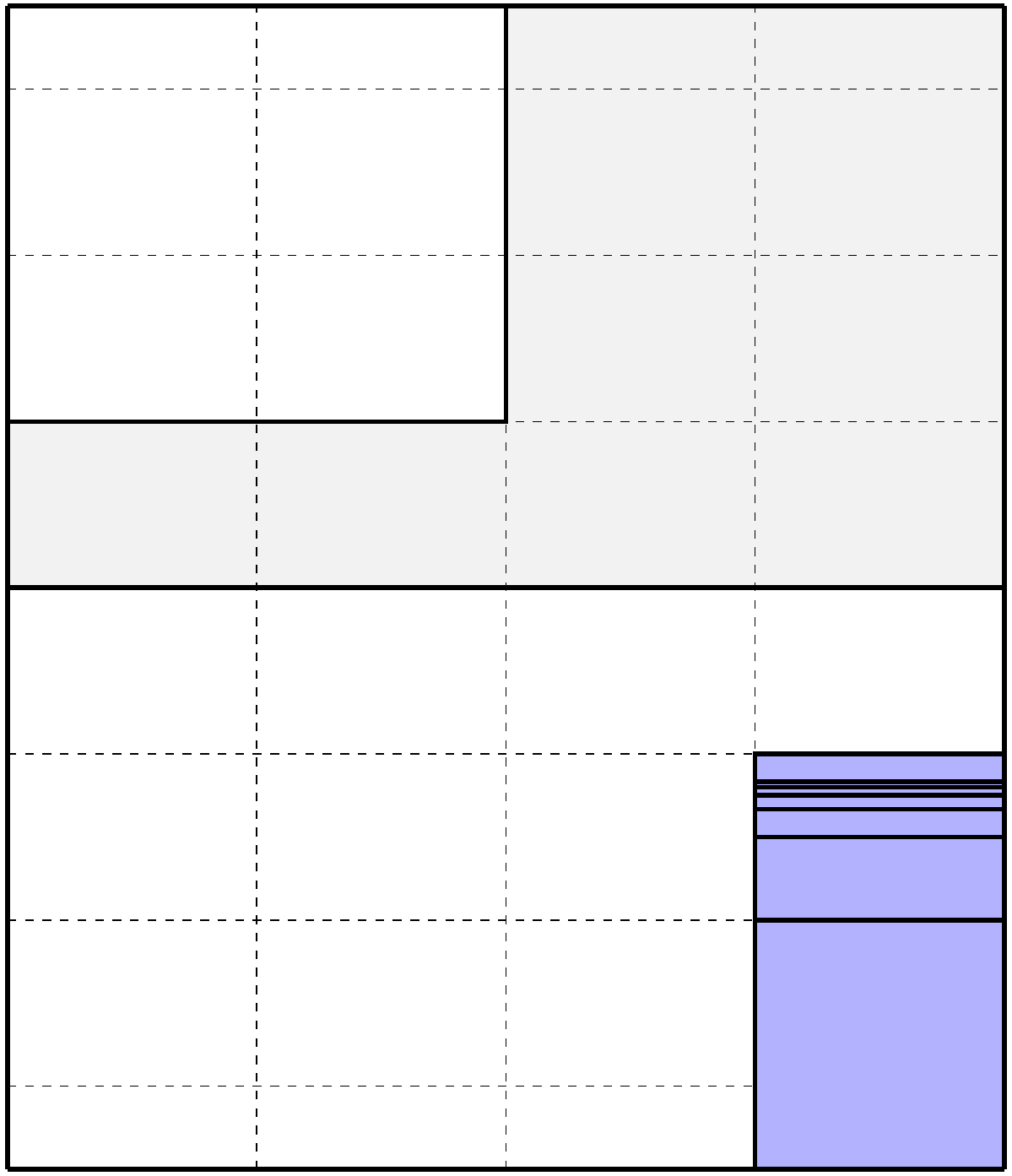}
\includegraphics[width=5cm]{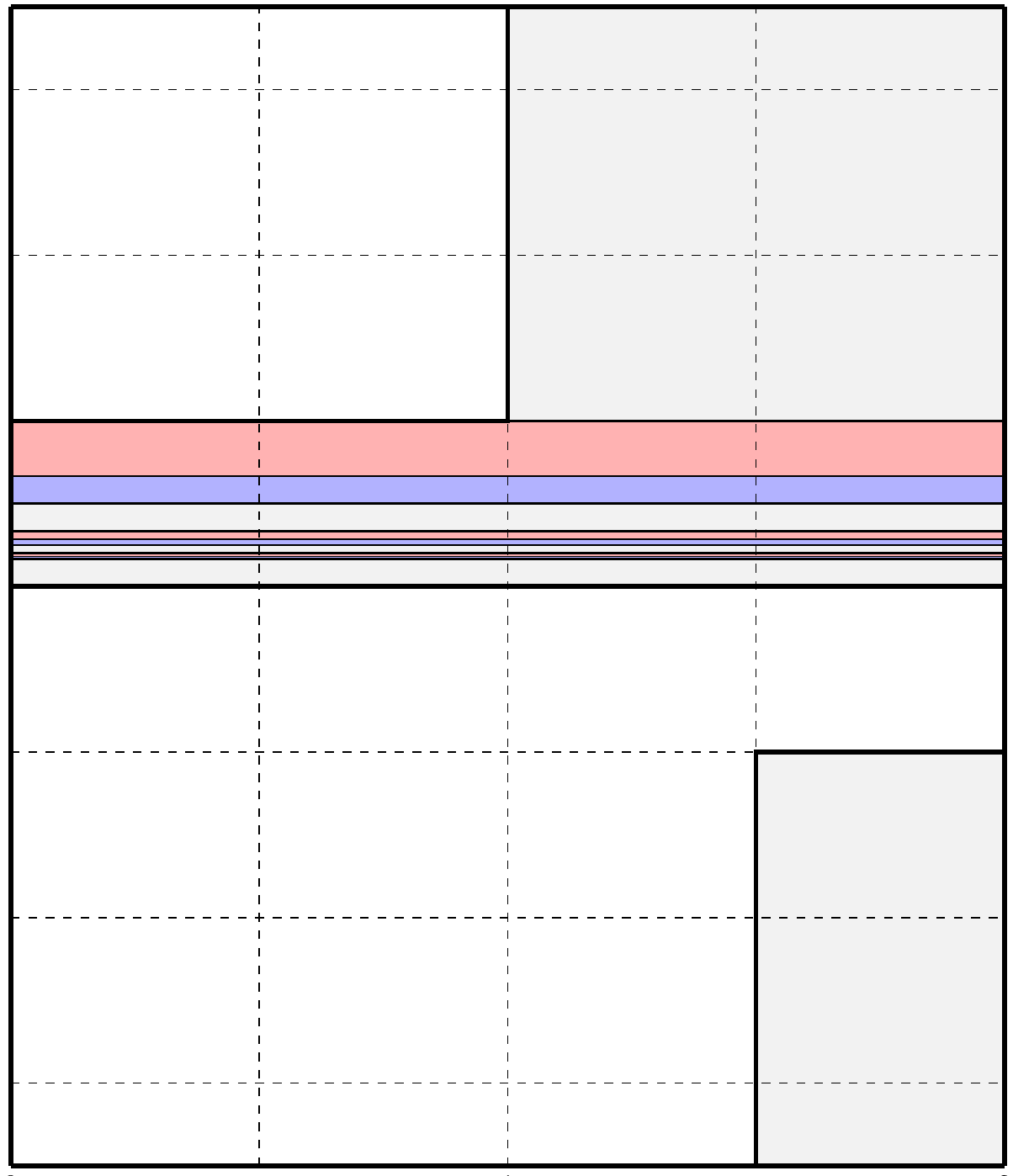}
\end{figure}
\end{proof}

%%%%%%%%%%%%%%%%%%%%%%%%%%%%%%%%%%%%%%%%%%%%
%%%%%%%%%%%%%%%%%%%%%%%%%%%%%%%%%%%%%%%%%%%%

%%%%%%%%%%%%%%%%%%%%%%%%%%%%%%%%%%%%%%%%%%%%
%%%%%%%%%%%%%%%%%%%%%%%%%%%%%%%%%%%%%%%%%%%%
\section{Lyapunov exponent and cocycles}
%%%%%%%%%%%%%%%%%%%%%%%%%%%%%%%%%%%%%%%%%%%%
%%%%%%%%%%%%%%%%%%%%%%%%%%%%%%%%%%%%%%%%%%%%

\subsection{Background on Lyapunov exponent}
%%%%%%%%%%%%%%%%%%%%%%
%%%%%%%%%%%%%%%%%%%%%%

We refer to \cite{Bochi.Viana.03}.

\begin{definition} \label{def:cocycles}
{\bf A cocycle} of the dynamical system $(X,T)$ is a map $M:X\times \mathbb{N}\rightarrow GL_2(\mathbb{R})$ such that
\begin{itemize}
\item $M(x,0)=Id$ for all $x\in X$,
\item $M(x,n+m)=M(T^n(x),m)M(x,n)  $ for all $x\in X$ and $n,m\in\mathbb{N}$.
\end{itemize}
\end{definition}

\begin{theoreme}[Oseledets]\label{thm:oseledets}
Let $(X,T)$ a dynamical system and $\mu$ be an invariant probability measure  for this system. 
Let $M$ be a cocycle  over $T$ such that for each $n\in\mathbb{N}$ the maps $x\mapsto \ln{|| M(x,1)||}, x\mapsto \ln{|| M(x,-1) ||}$ are $L^1$ integrable with respect to $\mu$. 

Then there exists a measurable map $Z$ from $X$ to $\mathbb R^2\setminus \{0\}$ and two functions $\lambda_+$ and $\lambda_-$ from $X$ to $\mathbb R$ which are $T$-invariant, such that $\lambda_+\geq \lambda_-$ and for $\mu$ almost all $x$ and for every non zero vector $z\in \mathbb{R}^2$ 
\[
\left\{ 
\begin{array}{ll}
\lim \frac{1}{n}\ln ||M(x,n)z||  \underset{n\to +\infty} \longrightarrow  \lambda_-(x) & \mbox{ if } z\in \mbox{ vect}\Big(Z(x)\Big) ,\\
\lim \frac{1}{n}\ln ||M(x,n)z||  \underset{n\to +\infty} \longrightarrow  \lambda_+(x) & \mbox{ if } z\notin \mbox{ vect}\Big(Z(x)\Big).
\end{array}
\right.
\]
\end{theoreme}

The numbers $\lambda_{\pm} (x)$ are called {\bf Lyapunov exponents} of the cocycle, see \cite{Oseledets} and \cite{Furst-Kest}.

\begin{theoreme}[Furstenberg-Kesten]\label{thm:FK}
Let $(X,\mu,T)$ a dynamical system and $\mu$ be an invariant probability measure  for this system. 
Suppose that $ \log{|| M(x,1)||}$ and $ \log{|| M(x,-1)||}$ are $L^1(\mu)$ integrable.
Then, for $\mu$ almost every $x$
\[
\lambda_+(x) = \lim \limits_n \frac1n \ln \| M(x,n)\| \quad  \mbox{ and } \quad   \lambda_-(x) = \frac1n   \lim \limits_n \ln \| M(x,n)^{-1}\|^{-1}.
\]

The functions $\lambda_+$ and $\lambda_-$ are $T$-invariant and
\[
\int \lambda_+ d \mu = \lim \limits_n \frac1n   \int  \ln \| M(x,n)\| d\mu(x)  \quad  \mbox{ and } \quad  \int  \lambda_- d\mu = \lim \limits_n  \frac1n  \int \ln \| M(x,n)^{-1}\|^{-1}d\mu(x)  .
\]

\end{theoreme}

\begin{corollaire} \label{co:FK}
If $\mu$ is an ergodic measure for $(X,T)$, then $\lambda_+$ and $\lambda_-$ are constant $\mu$ almost everywhere and 
\[
 \lambda_+  = \lim \limits_n \frac1n  \int  \ln \| M(x,n)\| d\mu(x)  \quad  \mbox{ and } \quad   \lambda_-  = \lim \limits_n \frac1n  \int \ln \| M(x,n)^{-1}\|^{-1}d\mu(x)  .
\]
\end{corollaire}

\subsection{Special case of cocyles in positive matrices in $GL_2(\mathbb Z)$} \label{subsecasparticulier}
%%%%%%%%%%%%%%%%%%%%%%%
%%%%%%%%%%%%%%%%%%%%%

Assume now that each matrix has non negative coefficients. Then it is clear that for almost every $x$ :
\begin{equation} \label{equationconepos}
\mbox{ vect} \Big(Z(x)\Big) \cap \Big( \mathbb R_+^*\Big) ^2 = \emptyset. 
\end{equation}

Recall the following metric on $\mathbb P\mathbb R^2$ given by
$$d\left(\begin{pmatrix}x\\y\end{pmatrix},  \begin{pmatrix}z\\t\end{pmatrix}\right)=\log{\frac{\max{(x/z,y/t)}}{\min{(x/z,y/t)}}}.$$
Now if $M$ is a positive matrix then the {\bf Lipschitz constant} of $M$ is defined by the following formula and is less than $1$, see \cite{Birkh.57} (p220):
$$L(M)=\max \frac{d(Mu,Mv)}{d(u,v)}= \mbox{tanh}\left(\frac{1}{4}\ln\left(\frac{ad}{bc}\right)\right)\leq1.$$
 Thus we deduce, see also \cite{Senet.06}.
 
 \begin{lemme} \label{ljsehc}
Let $(M_n)_n$ be a sequence of positive matrices.
If $\lim_nL(M_1\dots M_n)=0$ then there exists $Z$ such that for every $z\in \mathbb R_+^2$ the sequence $M_1\dots M_nz$ converges to $Z$.
 \end{lemme}

\begin{theoreme} \label{conv-prod-matr}
Let $(X,\mu,T)$ a dynamical system and $\mu$ be an ergodic probability measure  for this system. 
Let $MX  \to GL_2(\mathbb Z)\cap \mathcal M _2(\mathbb N)$ be a measurable map.

We suppose that for almost each $x \in X$, $M(x)$ is hyperbolic.

Then, there exists a measurable map $x \to z_x$ such that for almost each  for each $z\in \mathbb R_+^2$,
\[
\lim_{n\to+\infty}M(x)\times \cdots \times M(T^n x)(\mathbb R_+^2)= \mathbb R_+ \cdot z_x
\]
\end{theoreme}

We recall that a matrix is hyperbolic if $|tr(M)|> 2$.

Remark that a compact subgroup of $GL_2(\mathbb R)$ is included up to conjugation in $O(2)$. 
Thus we can replace one hypothesis by
for each integer $n$  $M_n$ is a positive matrix and
\begin{equation} \label{equationzekhk}
\mu \Big( x\in X ; M(x) \mbox{ is hyperbolic} \Big) >0.
\end{equation}

\subsection{Cocycles over the system $(\Omega,S,\nu)$ and over $(\Omega,\S,\boldsymbol \nu)$} \label{section:cocycleover}
%%%%%%%%%%%%%%%%%%%%%%%%%%%%%%%%%
%%%%%%%%%%%%%%%%%%%%%%%%%%%%%%%%%%%%%

The next Lemma explains why we do not work with the system $(\Omega,S,\nu)$ : 
\begin{lemme} \label{lemme:cocyclemarchepas}
The maps $\omega \mapsto \ln{\left \|  M(\omega)  \right\|}$ and $\omega \mapsto \ln{\left \|   M(\omega)^{-1}  \right\|}$ are not L$^1$ integrable with respect to $\nu$. 
\end{lemme}

\begin{proof} 
The matrices $M(\omega)$ are in $\mathcal M_2(\mathbb N)$ and then $\ln\| M_\omega\|_1$ is positive.

We consider the identification between $\Omega$ and $[0,2]$.
Then $\ln\|  M(\omega)\|_1=\ln(2)$ if $\omega\in[1,3/2]$ and a density of $\nu$ is $\frac{1}{x(x-1)}$ on $[1,3/2]$.
\end{proof}

\begin{lemme}\label{lem:integration}
We have:
\begin{enumerate}
\item The function $\omega \to  \ln \r $ is in L${}^1(\boldsymbol \nu)$.   
\item The functions
$\omega \to \ln \|  \M(\omega) \| $ and $\omega \to \ln \|  \M(\omega)^{-1} \| $ are in L${}^1(\boldsymbol \nu)$. 
\end{enumerate}
Moreover
\begin{equation} \label{equationintegraleretM}
 \int \ln \| \M(\omega) \|_\infty d \boldsymbol \nu (\omega) \leq  3.8 
\quad \mbox{ and } \quad 
2.46 \leq \int \ln \r(\omega)   d \boldsymbol \nu (\omega) \leq  2.47.
\end{equation}
\end{lemme}

\begin{proof}
We find 
\[
\begin{array}{ll}
\displaystyle \int \ln \| \M(x) \|_\infty d \boldsymbol \nu (x) 
	&\displaystyle
		 = \int \ln  \circ \max \Big( \m_{1,1}(x)+\m_{1,2}(x),\m_{2,1}(x)+\m_{2,2}(x)\Big)  d \boldsymbol \nu (x) \\
	&\displaystyle
		 =\sum \limits_{n\geq 1} \int_{ \frac{1}{n+1} }^{ \frac{1}{n} }  \ln  \circ \max \Big(2n+1,n+1)\Big)  \frac{1}{x+1} d x \\
	& \qquad \qquad \displaystyle		
		+ \sum \limits_{n\geq 2}  \int_{ 1+\frac{1}{n+1} }^{ 1+\frac{1}{n} } \ln  \circ \max \Big(2n-1,1 \Big)  \frac{1}{x} d x \\
	& \qquad \qquad \displaystyle		
		+ \sum \limits_{n\geq 2}  \int_{2- \frac{1}{n} }^{ 2-\frac{1}{n+1} } \ln  \circ \max \Big(2n+1,n\Big)  \frac{1}{x-1} d x \\
	&\displaystyle
		 =\sum \limits_{n\geq 1}  \ln (2n+1)  \int_{ \frac{1}{n+1} }^{ \frac{1}{n} }   \frac{1}{x+1} d x 
		+ \sum \limits_{n\geq 2}  \ln (2n-1)   \int_{ 1+\frac{1}{n+1} }^{ 1+\frac{1}{n} }   \frac{1}{x} d x \\
	& \qquad \qquad	 \displaystyle	
		+ \sum \limits_{n\geq 2}   \ln (2n+1)  \int_{2- \frac{1}{n} }^{ 2-\frac{1}{n+1} }\frac{1}{x-1} d x .
\end{array}
\]

and
\[
\begin{array}{ll}
\displaystyle \int \ln \r(x)   d \boldsymbol \nu (x) 
	&\displaystyle
		=\int_0^1 \ln \r(x) \frac{1}{x+1} dx +\int_1^{3/2} \ln \r(x) \frac{1}{x} dx +\int_{3/2} ^2 \ln \r(x) \frac{1}{x-1} dx \\
	&\displaystyle
		=- \int_0^1  \frac{\ln (x)}{x+1} dx +\sum \limits_{n\geq 2}  \int_{ 1+\frac{1}{n+1} }^{ 1+\frac{1}{n} } \ln \r(x) \frac{1}{x} dx - \int_{3/2} ^2   \frac{\ln(2-x)}{x-1} dx \\
	&\displaystyle
	=- \int_0^1  \frac{\ln (x)}{x+1} dx - \sum \limits_{n\geq 2}  \int_{ 1+\frac{1}{n+1} }^{ 1+\frac{1}{n} }  \frac{\ln\big(n-(n-1)x \big)}{x} dx - \int_{3/2} ^2   \frac{\ln(2-x)}{x-1} dx. \\
\end{array}
\]

\end{proof}

\end{document}